\documentclass[abstract=true,10pt,DIV=12]{scrartcl}
\KOMAoption{numbers}{noendperiod}
\addtokomafont{disposition}{\rmfamily}

\usepackage{lmodern}
\usepackage[british]{babel}
\usepackage[utf8]{inputenc}
\usepackage[T1]{fontenc}
\usepackage[english=british]{csquotes}

\usepackage
[
	backend=biber,
	language=british,
	sorting=none,
    sortcites=true,
	sortlocale=en-GB,
    giveninits=true,
    maxnames=4,
    autocite=superscript
]{biblatex}
\DeclareCiteCommand{\supercite}[\mkbibsuperscript]{%
\iffieldundef{prenote}{}{\BibliographyWarning{Ignoring prenote argument}}%
\iffieldundef{postnote}{}{}}%
{\bibopenbracket\usebibmacro{citeindex}\usebibmacro{cite}\usebibmacro{postnote}\bibclosebracket}{\supercitedelim}{}

\usepackage{graphicx}
\graphicspath{{./pictures/}}
\DeclareGraphicsExtensions{.pdf,.png}

\usepackage{listings}
\usepackage{booktabs}
\usepackage{amsmath} 
\usepackage{amssymb}
\usepackage{mathtools}
\usepackage{pgfplots}
\pgfplotsset{compat=1.14}
\usepgfplotslibrary{groupplots}
\usepackage{subcaption}

\usepackage[amsmath,thmmarks,hyperref]{ntheorem}
\theorempreskip{1ex plus .25ex minus .1ex}
\theorempostskip{1ex plus .25ex minus .1ex}

\theoremstyle{nonumberplain}
\theoremheaderfont{\normalfont\itshape}
\theorembodyfont{\normalfont}
\theoremseparator{.}
\theoremsymbol{$\square$}
\newtheorem{proof}{Proof}

\theoremstyle{plain}
\theoremheaderfont{\normalfont\scshape}
\theorembodyfont{\normalfont\itshape}
\theoremseparator{.}
\theoremsymbol{}
\newtheorem{theorem}{Theorem}[section]
\newtheorem{lemma}[theorem]{Lemma}

\theoremheaderfont{\normalfont\itshape}
\theorembodyfont{\normalfont\small}

\usepackage{xcolor}
\usepackage{hyperref}
\usepackage{hyperxmp}
\colorlet{internallinkcolour}{green!50!black}
\colorlet{externallinkcolour}{red!50!black}

\hypersetup
{
  pdfauthor          = {Matthias Kirchhart; Christian Rieger},
  pdftitle           = {A Particle Method without Remeshing},
  pdfkeywords        = {particle methods; numerical analysis; computational fluid dynamics; level-set method},
  pdfencoding        = {unicode},
  pdflang            = {en-GB},
  pdfdisplaydoctitle = {true},
  unicode            = {true},
  colorlinks         = {true}, 
  allcolors          = internallinkcolour,
  urlcolor           = externallinkcolour,
  frenchlinks        = false,
  pdfborder          = {0 0 0},
  naturalnames       = false,
  hypertexnames      = false,
  breaklinks         = true
}

\usepackage[nameinlink,noabbrev]{cleveref}
\crefname{equation}{}{}

\crefname{assumption}{assumption}{assumptions}
\Crefname{assumption}{Assumption}{Assumptions}
\crefname{subsection}{subsection}{subsections}
\Crefname{subsection}{Subsection}{Subsections}

\newcommand{\bigO}[1]{\mathcal{O}\left({#1}\right)}
\newcommand{\spdim}{\mathrm{D}}   
\newcommand{\reals}{\mathbb{R}}
\newcommand{\wholespace}{\reals^\spdim}

\newcommand{\pd}{\partial}
\newcommand{\ddx}[2]{\frac{{\mathrm d}{#1}}{{\mathrm d}{#2}}}

\newcommand{\pdx}[2]{\frac{\partial {#1}}{\partial {#2}}}
\newcommand{\idx}[4]{\int_{#1}^{#2} {#3}\,{\mathrm d}{#4}}
\newcommand{\boldnabla}{\boldsymbol{\nabla}}
\newcommand{\boldlaplace}{\boldsymbol{\Delta}}

\newcommand{\vv}[1]{\mathbf{#1}}        
\newcommand{\gv}[1]{\boldsymbol{#1}}    

\DeclareMathOperator{\diam}{diam}

\DeclareMathOperator{\meas}{meas}

\newcommand{\traj}[2]{\gv{\Phi}_{\kern -1pt #1}^{\kern -1pt #2}}
\newcommand{\sci}[2]{\ensuremath{{#1}\times10^{#2}}}

\usepackage{microtype}

\begin{document}

\renewcommand*{\thefootnote}{\fnsymbol{footnote}}
\title{A Particle Method without Remeshing}
\author{Matthias Kirchhart\footnote{Centre for Computational Engineering Science,
Mathematics Division, RWTH Aachen University. Schinkelstraße~2, 52062~Aachen,
Germany. E-Mail: \url{kirchhart@mathcces.rwth-aachen.de}}\and Christian Rieger%
\footnote{Institute for Numerical Simulation, University of Bonn. Endenicher
Allee~196, 53115~Bonn, Germany. \url{rieger@ins.uni-bonn.de}}}
\date{}
\maketitle

\vspace{-2\baselineskip}
\begin{abstract}\noindent
We propose a simple tweak to a recently developed regularisation scheme for
particle methods. This allows us to choose the particle spacing $h$ proportional
to the regularisation length $\sigma$ and achieve optimal error bounds of the
form $\mathcal{O}(\sigma^n)$, $n\in\mathbb{N}$, without any need of remeshing.
We prove this result for the linear advection equation but also carry out high-%
order experiments on the full Navier--Stokes equations. In our experiments the
particle methods proved to be highly accurate, long-term stable, and competitive
with discontinuous Galerkin methods.\\[1em]
\end{abstract}

\section{Introduction}
The convergence of many classical, uniform discretisations of partial differential
equations is often governed by only two parameters: the discretisation order
$n\in\mathbb{N}$ and the underlying mesh-size $\sigma>0$. Assuming that the exact
solution is smooth enough, one typically obtains error bounds of the type
$\bigO{\sigma^n}$.\autocite[Theorem~(5.4.8)]{brenner2008} Particle methods like
Smoothed Particle Hydrodynamics~(SPH) or Vortex Methods~(VM), on the other hand,
feature \emph{two} orders $n,m\in\mathbb{N}$ and \emph{two} sizes $h,\sigma>0$.
Here, $h$ describes some form of particle spacing and $m$ the order of an underlying
quadrature rule, while on the other hand $\sigma$ describes a smoothing length
and $n$ the order of a regularisation scheme. These parameters need to be very
carefully chosen to ensure convergence. Typical error estimates for solutions to,
e.\,g., the linear advection equation read:\autocite[Theorem~4.2]{raviart1985}
\begin{equation}\label{eqn:classicbound}
\Vert u-u_{h,\sigma}\Vert_{L^p(\Omega)} =
\bigO{\sigma^n+\biggl(\frac{h}{\sigma}\biggr)^m}.
\end{equation}
For fixed values of $n$, $m$, and $\sigma$, the optimal choice of $h$ thus is a
compromise that balances both contributions. In the case $n=m$, this results in
$h\sim\sigma^2$; in other words the particle spacing $h$ needs to be negligible
compared to the smoothing length. In practice this often is prohibitively
expensive. An alternative is to \enquote{cheat} and set $m=\infty$: on first
sight this would allow for choosing $h\sim\sigma^{1+\varepsilon}$ for any
$\varepsilon>0$, i.\,e., $h$ could essentially be chosen proportional to $\sigma$.
However, the constants hidden in the $\mathcal{O}$-notation very quickly grow with
$m$ and time $t$, and thus also make this approach infeasible in practice. For
this reason current particle methods typically have to \emph{remesh} the
particles to their original locations after every other time-step or so. But
particle approximations are \emph{exact solutions} of the advection equation,
in a sense that will be made clear in \Cref{subsec:particleapprox}. Apart from
destroying the purely Lagrangian character of the method, the remeshing process
deviates from this exact solution and thereby introduces further errors.

In this work we describe a surprisingly simple tweak to a recently developed
regularisation scheme, enabling us to chose $h\sim\sigma$ and $m$ independent of
$n$ while still yielding the optimal error bound of $\bigO{\sigma^n}$. In our
numerical experiments this difference turned out to be dramatic and enabled us
to perform long-term simulations without remeshing.

The rest of this article is structured as follows. In \cref{sec:particlemethods}
we review \enquote{classical} particle methods and some of the key results from
their analysis. In \cref{sec:literature} we give references to the literature
for further reading and related results. \Cref{sec:splineapproach} is the core
of this article: we describe a simple tweak to a recently proposed scheme:
particle regularisation by projection onto spline spaces. A complete convergence
analysis for the linear advection equation is provided. Finally, in
\cref{sec:experiments} we carry out numerical experiments for the full, non-linear
Navier--Stokes equations. For simplicity, throughout this article we restrict
ourselves to the geometry of an axis-aligned cube, but point out that the
regularisation scheme also generalises to the case of domains with arbitrary
Lipschitz boundary. We conclude with some remarks on remaining open problems 
and possible future extensions.

\section{Particle Methods}\label{sec:particlemethods}
In this section we recall the necessary ingredients to describe particle
methods. For a rigorous derivation of these results the reader is referred to
Raviart's lecture notes.\autocite{raviart1985}

\subsection{Linear Advection Equation}
In order to keep our focus on the problem at hand, we will discuss particle
methods in one of the simplest possible settings: the linear advection equation
in the box. Thus let $\Omega:=(0,1)^\spdim$, typically $\spdim\in\lbrace 2,3\rbrace$.
Let $\vv{a}:\Omega\times[0,T]\to\wholespace$ denote a given smooth and bounded
velocity field that---for simplicity---also satisfies $\nabla\cdot\vv{a}\equiv 0$
and $\vv{a}\cdot\vv{n} = 0$ on $\pd\Omega$. Given continuous initial data $u_0\in
C(\Omega)$, we are looking for the solution $u:\Omega\times[0,T]\to\reals$ of the
following initial value problem:
\begin{equation}\label{eqn:advequation}
\begin{split}
\left\lbrace
\begin{aligned}
\pdx{u}{t} + (\vv{a}\cdot\nabla)u &= 0            &    \text{in }&\Omega\times[0,T], \\
u(\vv{x},0)                       &= u_0(\vv{x})  &    \text{on }&\Omega.
\end{aligned}
\right.
\end{split}
\end{equation}
Throughout this article we demand continuity of the initial data, such that
point-wise evaluation is well-defined. It is well-known that problem
\eqref{eqn:advequation} can be solved by the method of characteristics. Thus,
for any $(\vv{x},\tau)\in\Omega\times[0,T]$ let us define the \emph{trajectory}
$\vv{X}(t;\vv{x},\tau)$ as the solution of the following initial value problem:
\begin{equation}
\begin{split}
\left\lbrace
\begin{aligned}
\ddx{\vv{X}}{t}(t;\vv{x},\tau) &= \vv{a}\bigl(\vv{X}(t;\vv{x},\tau),t\bigr), \\
\vv{X}(\tau;\vv{x},\tau) &= \vv{x}.
\end{aligned}
\right.
\end{split}
\end{equation} 
By the Picard--Lindelöf theorem one obtains that $\vv{X}$ is well-defined.
Moreover, it can be shown that the map $\traj{\tau}{t}(\vv{x}) := \vv{X}(t;
\vv{x},\tau)$ is a diffeomorphism with inverse $\traj{t}{\tau}$. Colloquially
speaking $\traj{\tau}{t}(\vv{x})$ tells us where the particle with position
$\vv{x}$ at time $\tau$ will be at another time $t$. The solution to the
advection equation~\eqref{eqn:advequation} then is given by $u(\vv{x},t)=
u_0\bigl(\traj{t}{0}(\vv{x})\bigr)$.\autocite[Equation~(1.11)]{raviart1985}

In the lecture notes of Raviart\autocite{raviart1985} the theory is then
extended to weak solutions in Sobolev spaces and it is shown that the solutions
are bounded, i.\,e.:
\begin{equation}\label{eqn:advstability}
\Vert u(\cdot,t)\Vert_{W^{s,p}(\Omega)}\lesssim\Vert u_0\Vert_{W^{s,p}(\Omega)}
\qquad t\in [0,T],\ s\in\mathbb{R},\ p\in[1,\infty].
\end{equation} 
Throughout this article the symbol $C$ refers to a generic constant $C>0$ that
is independent of $h$, $\sigma$, and the functions occurring in the norms. The
notation $a\lesssim b$ means $a\leq Cb$ for some $C>0$, and $a\sim b$ means
$a\lesssim b\lesssim a$. In estimate~\eqref{eqn:advstability}, for example, the
hidden constant depends on $\vv{a}$ and $T$, but is independent of $u(\cdot,t)$,
$u_0$, and $t$.

\subsection{Particle Approximations}\label{subsec:particleapprox}
A simple, intuitive approach to numerically solving the advection equation~%
\eqref{eqn:advequation} could consist of storing samples $u_i := u_0(\vv{x}_i)$
of the initial data $u_0$ at a finite set of locations $\vv{x}_i\in\Omega$, $i=1,
\dotsc,N$. One could then track these \emph{particles} over time by solving
$\dot{\vv{x}}_i(t) = \vv{a}(\vv{x}_i(t),t)$ by means of, e.\,g., a Runge--Kutta
method. We then know that at any time $t$ we have $u(\vv{x}_i(t),t) = u_i$. The
question that then arises, however, is what happens in-between the particles.
One could, of course, devise interpolation schemes, but these do not easily
generalise to general bounded domains.

A related idea involves quadrature rules. Let us subdivide the domain $\Omega=
(0,1)^\spdim$ into uniform squares/cubes of edge length $h$. To each of these
cells we can apply a quadrature rule of polynomial exactness degree $m-1$ with
positive weights, e.\,g., Gauß--Legendre rules. This yields a set of nodes
$\vv{x}_i$ with associated weights $w_i$, $i=1,\dotsc,N$. Furthermore setting
$u_i:=u_0(\vv{x}_i)$, $i=1,\dotsc,N$, and letting $\delta$ denote the Dirac
delta distribution, the functional $u_h(t=0):=\sum_{i=1}^{N}w_iu_i\delta_{\vv{x}_i}$
approximates $u_0$ in the sense that for smooth functions $\varphi$ it holds
that:
\begin{equation}
\langle u_h(0),\varphi\rangle = \sum_{i=1}^N w_iu_i\varphi(\vv{x}_i) 
\approx
\idx{\Omega}{}{u_0\varphi}{\vv{x}}.
\end{equation}

This approach has the advantage that error-bounds are readily available. To
avoid technicalities, let us for simplicity assume that $m\geq\spdim$, such
that by the Sobolev embedding theorem functions from $W^{s,p}(\Omega)\hookrightarrow
C(\Omega)$ are continuous for all $m\geq s\geq\spdim$, $p\in[1,\infty]$. The
following bound then is a simple consequence of the Bramble--Hilbert lemma%
\autocite[Lemma~(4.3.8)]{brenner2008}:
\begin{equation}\label{eqn:quaderr}
\Vert u_h(0) - u_0\Vert_{W^{-s,p}(\Omega)}\lesssim h^s\Vert u_0\Vert_{W^{s,p}(\Omega)}
\qquad m\geq s\geq\spdim,\ p\in[1,\infty].
\end{equation}

The \emph{exact solution} to the advection equation~\eqref{eqn:advequation} with $u_0$
replaced by $u_h(0)$ is again given by moving the particles according to
$\dot{\vv{x}}_i(t)=\vv{a}(\vv{x}_i(t),t)$. Together with the stability estimate~%
\eqref{eqn:advstability} this immediately yields:
\begin{equation}
\Vert u_h(t) - u(t)\Vert_{W^{-s,p}(\Omega)}\lesssim h^s\Vert u_0\Vert_{W^{s,p}(\Omega)}
\qquad m\geq s\geq\spdim,\ p\in[1,\infty].
\end{equation}
The problem here, however, is even worse. In fact, $u_h$ is an irregular
distribution that cannot be interpreted as an ordinary function. Even at the
particle locations, we strictly speaking do not have function values, but only
weights $w_iu_i$.

On paper the two approaches are of course somehow equivalent: it is trivial
to obtain the function value $u_i$ from the weight $w_iu_i$ and vice versa. But
on the one hand, the first approach yields function values but does not allow us
to perform numerical integration, while in the second approach the situation is
reversed.

\subsection{Particle Regularisation}\label{subsec:regularisation}
Particle regularisation refers to the process of obtaining a function
$u_{h,\sigma}$ from a given particle approximation $u_h$ that is interpreted as
a quadrature rule. The most common approach uses mollification and is more
easily explained in the whole-space case $\Omega=\wholespace$. Let $\zeta:
\wholespace\to\reals$ be a smooth function that fulfils $\idx{\wholespace}{}{%
\zeta(\vv{x})\vv{x}^\alpha}{\vv{x}} = \vv{x}^\alpha|_{\vv{x}=0}$ for all multi-%
indeces $|\alpha|<n$ and some fixed $n\in\mathbb{N}$. In other words, for
$|\alpha|<n$, convolution of $\vv{x}^\alpha$ with $\zeta$ behaves like convolution
with the Dirac delta distribution: $\bigl(\zeta\star\vv{x}^\alpha\bigr) =
\bigl(\delta\star\vv{x}^\alpha\bigr)= \vv{x}^\alpha$ and in this sense
$\zeta\approx\delta$. Furthermore, for stability, it should hold that
$\idx{\wholespace}{}{|\vv{x}|^n|\zeta(\vv{x})|}{\vv{x}} < \infty$. A multitude
of such \emph{kernel} or \emph{blob functions} is available in the literature.%
\autocite[Section~2.3]{cottet2000} After choosing $\zeta$ and some $\sigma>0$
one scales $\zeta_\sigma(\vv{x}):= \sigma^{-\spdim}\zeta(\tfrac{\vv{x}}{\sigma})$.
Finally, because $u_h\approx u$ and $\zeta_\sigma\approx\delta$ one obtains $u =
u\star\delta \approx u_h\star\zeta_\sigma =: u_{h,\sigma}$, that is:
\begin{equation}
u_{h,\sigma}(\vv{x},t):= \sum_{i=1}^{N}w_iu_i\zeta_\sigma\bigl(\vv{x}-\vv{x}_i(t)\bigr).
\end{equation}
A careful analysis then reveals the aforementioned error bound%
\autocite[Theorem~4.2]{raviart1985}:
\begin{equation}
\Vert u(\cdot,t)-u_{h,\sigma}(\cdot,t)\Vert_{L^p(\Omega)}\lesssim\biggl(
\sigma^n + \biggl(\frac{h}{\sigma}\biggr)^m
\biggr)\Vert u_0\Vert_{W^{\max{\lbrace n,m\rbrace},p}(\Omega)}.
\end{equation}
The origin of the $\left(\tfrac{h}{\sigma}\right)^m$-term lies in the quadrature
error estimate~\eqref{eqn:quaderr}. On the left the error is measured in the
$W^{-m,p}$ norm, on the right we have the $W^{m,p}$ norm, i.\,e., a difference of
$2m$ orders. Yet, the quadrature error is only $\bigO{h^m}$ as opposed to
$\bigO{h^{2m}}$. Because of Bakhvalov's theorem,\autocite{sobolev1997} the
$\bigO{h^m}$ bound is asymptotically optimal and cannot be improved. This
ultimately forces one to choose $h\ll\sigma$.

\section{The Proposed Method in Context of the Literature}\label{sec:literature}
Vortex methods are the oldest particle methods and can at least be traced back
to the early 1930s, when Rosenhead tried to numerically answer the question
whether vortex sheets roll up.\autocite{rosenhead1931} The first regularised
vortex methods appeared much later in the early 1970s due to Chorin\autocite{chorin1973},
who used a blob-based regularisation, and Christiansen\autocite{christiansen1973},
who used a grid-based regularisation on simple, rectangular domains. The underlying
ideas of particle methods have been re-introduced in different contexts at least
three times: in the late 1950s Harlow\autocite{harlow1957} and Evans and Harlow%
\autocite{harlow1957b} introduced the Particle-in-Cell~(PIC) method. The mapping
between particle and grid quantities is a regularisation step, though this fact
is not emphasised in these works. Lucy\autocite{lucy1977} as well as Gingold and
Monaghan\autocite{gingold1977} laid the ground for Smoothed Particle Hydrodynamics~(SPH),
making use of a blob-based regularisation.

Vortex particle methods using the blob-regularisation were first analysed by
Dushane\autocite{dushane1973} and later by Hald and Mauceri Del Prete\autocite{%
hald1978} and Hald\autocite{hald1979}. Many contributions followed their work,
and we refer to Leonard\autocite{leonard1980,leonard1985} for historic comments.
Later it was realised that particle approximations that are interpreted as a
quadrature rule correspond to \emph{exact solutions} of a weak formulation of
the transport equation. This lead to a new, simplified type of convergence proofs
due to Raviart\autocite{raviart1985} and Cottet\autocite{cottet1988}. To our
knowledge, the latter work also contains the first convergence proofs for
vortex methods using a grid-based regularisation. In the year 2000 Cottet and
Koumoutsakos published the first monograph on vortex methods.\autocite{cottet2000}
This work also contains many more historic remarks and an extensive bibliography.
Recently, we proposed regularisation schemes based on the $L^2$-projection onto
finite element and spline spaces with similar error bounds that also work in
general, bounded domains.\autocite{kirchhart2017b,kirchhart2019} These techniques
are reminiscent of the earlier FEM-blobs suggested by Merriman.\autocite{merriman1991}
All of these analyses feature typical error-bounds like the one in equation~\cref{%
eqn:classicbound}.

In the context of vortex methods, remeshing was introduced by Koumoutsakos.%
\autocite{koumoutsakos1997}  It now is ubiquitous in practice,\autocite{%
koumoutsakos2008,mimeau2017,gillis2019} and has also been the subject of
numerical analysis\autocite{cottet2014}.

Cohen and Berthame\autocite{cohen2000} pointed out that, at least in principle,
the optimal convergence order $\bigO{\sigma^n}$ can be restored in particle
methods when considering function values $u_i$ instead of weights $w_iu_i$.
They devise a scheme that employs discontinuous, piece-wise polynomial
interpolations to achieve this error-bound. The triangulated vortex method of
Russo and Strain\autocite{russo1994} creates a triangulation of the domain using
the particle locations as grid-points. The particle field is then regularised by
using piece-wise linear interpolation on each triangle.

The approach discussed in this article is slightly similar in the regard that it
also uses function values $u_i$ and a finite element function space. However it
differs in the regard that it additionally makes use of the quadrature weights
$w_i$, does not require a triangulation that follows the grid-points, and easily
generalises to bounded domains and arbitrary order $n\in\mathbb{N}$. Moreover,
our method is conservative.

\section{Regularisation by Projection}\label{sec:splineapproach}
Assume we are given a continuous finite element space $V_\sigma^n(\Omega)\subset
C(\Omega)$ of mesh-width $\sigma$ and order~$n\in\mathbb{N}$. Furthermore assume 
that we are given a particle approximation $u_h = \sum_{i=1}^{N}w_iu_i\delta_{\vv{x}_i}$
that we want to regularise, where for brevity we sometimes omit the dependency on
time $t$ in our notation. Regularisation by projection now corresponds to finding
the solution $u_{h,\sigma}\in V_{\sigma}^n(\Omega)$ of the following system:
\begin{equation}
\idx{\Omega}{}{u_{h,\sigma}v_\sigma}{\vv{x}} = \sum_{i=1}^{N}w_iu_iv_\sigma(\vv{x}_i)
\qquad\forall v_\sigma\in V_{\sigma}^n(\Omega).
\end{equation}
Together with a fictitious domain approach, this idea generalises to arbitrary
domains $\Omega$, and a numerical analysis\autocite{kirchhart2017b,kirchhart2019}
reveals error bounds of the type~\eqref{eqn:classicbound}. The new approach
consists of replacing the \emph{exact integral} on the left by \emph{numerical
integration} using the nodes $\vv{x}_i$ and weights $w_i$ of $u_h$, i.\,e., we
instead find $u_{h,\sigma}\in V_\sigma^n(\Omega)$ such that:
\begin{equation}\label{eqn:neqapproach}
\sum_{i=1}^{N}w_i u_{h,\sigma}(\vv{x}_i)v_\sigma(\vv{x}_i) =
\sum_{i=1}^{N}w_iu_iv_\sigma(\vv{x}_i)
\qquad\forall v_\sigma\in V_{\sigma}^n(\Omega).
\end{equation}
Despite the additional error from discretising the integral, this approach does
in fact yield the desired error-bound of $\bigO{\sigma^n}$ for $h\sim\sigma$ and
arbitrary $m\geq 1$. The following three sub-sections are devoted to proving this
claim. Afterwards we discuss the relationship of this method to conventional
blob-based approaches.

\subsection{Spline Spaces}
For our ansatz spaces $V_\sigma^n(\Omega)$ we will use Cartesian tensor product
splines, although other conventional, $W^{1,\infty}(\Omega)$-conforming finite
element spaces would also be possible. Let us create a Cartesian grid of size
$\sigma>0$ for the domain $\Omega=(0,1)^\spdim$, whose cubes we will refer to as
$Q^\sigma_i\in\Omega$. For $n\in \mathbb{N}$ we then define our ansatz space as
follows:
\begin{equation}
V_\sigma^n(\Omega) := \bigl\lbrace u\in C^{n-2}(\Omega):\ 
u|_{Q^\sigma_i}\in\mathbb{Q}_{n-1},\ Q^\sigma_i\in\Omega\bigr\rbrace,
\end{equation}
where $\mathbb{Q}_{n-1}$ refers to the space of polynomials of \emph{coordinate-%
wise} degree $n-1$ or less. For $n=1$ one obtains the space of piecewise constants.
To ensure continuity, we will later restrict ourselves to $n\geq 2$.

It will sometimes be useful to specify the norm we employ on these spaces
explicitly. In these cases, for $p\in[1,\infty]$, we will write $V_\sigma^{n,p}
(\Omega)$ to refer the space $V_\sigma^n(\Omega)$ equipped with the $L^p(\Omega)$-%
norm. In the other cases the index $p$ will be omitted. Furthermore, in analogy
to the Sobolev Spaces, we will write $V_\sigma^{-n,p}(\Omega)$ for the normed dual
of $V_\sigma^{n,q}(\Omega)$, $\tfrac{1}{p} + \tfrac{1}{q}=1$.

We will assume that the reader is familiar with the basic properties of these
spaces, which are, e.\,g., described in great detail Schumaker's book\autocite{schumaker2007}.
We will in particular make use of the quasi-interpolator $P_\sigma^n: L^1(\Omega)
\to V_\sigma^n(\Omega)$, which has the following properties:\autocite[Theorems~12.6 and 12.7]{schumaker2007}
\begin{align}
\label{eqn:interpolantbound}
\Vert P_\sigma^n u\Vert_{W^{s,p}(Q_i^\sigma)} &\lesssim
\Vert u\Vert_{W^{s,p}(\hat{Q}_i^\sigma)} &
&\forall Q^\sigma_i\in\Omega, 0 \leq s\leq n-1,\ p\in[1,\infty],\\
\label{eqn:interpolanterror}
\Vert u - P_\sigma^nu\Vert_{L^p(Q_i^\sigma)} &\lesssim
\sigma^s\vert u\vert_{W^{s,p}(\hat{Q}_i^\sigma)} &
&\forall Q^\sigma_i\in\Omega,\ 0 \leq s\leq n,\ p\in[1,\infty],
\end{align}
where the hidden constants only depend on $n$, $\spdim$, and $s$. Here
$\hat{Q}_i^\sigma\supset Q_i^\sigma$ is only slightly larger than $Q_i^\sigma$,
in particular we have $\meas_\spdim(\hat{Q}_i^\sigma)\leq C(n,\spdim)\sigma^\spdim$.
$P_\sigma^n u$ also approximates the derivatives of $u$, but for simplicity we
will only discuss the $L^p(\Omega)$-norms here.

We will also make frequent use of so-called inverse estimates, which for
$v\in V_\sigma^n(\Omega)$ allow us to estimate stronger norms by weaker ones.%
\autocite[Section~4.5]{brenner2008} For this let $Q_i^\sigma$ denote an arbitrary
cube from the Cartesian grid. One then has locally:
\begin{equation}
\Vert u\Vert_{W^{s,p}(Q_i^\sigma)}\lesssim
\sigma^{\frac{\spdim}{p}-\frac{\spdim}{q}}\sigma^{r-s}
\Vert u\Vert_{W^{r,q}(Q_i^\sigma)}\qquad \forall\ p,q\in[1,\infty], 0\leq r\leq s\leq n-1,
\end{equation}
and globally:
\begin{equation}\label{eqn:globinverseineq}
\Vert u\Vert_{W^{s,p}(\Omega)}\lesssim
\sigma^{\min\lbrace 0,\frac{\spdim}{p}-\frac{\spdim}{q}\rbrace}\sigma^{r-s}
\Vert u\Vert_{W^{r,q}(\Omega)}\qquad \forall\ p,q\in[1,\infty], 0\leq r\leq s\leq n-1.
\end{equation}
Here the hidden constants only depend on $p$, $q$, $n$, $r$, $s,$ and $\spdim$.
The global inequality holds in general for any finite union of entire cubes $Q_i^\sigma$
from the Cartesian grid.

\subsection{Particle Approximation}\label{sec:particleapprox}
As mentioned before, for our purposes it is enough to consider quadrature rules
of order $m=1$. We create \emph{another} Cartesian grid of size $h\leq\sigma$.
Into each of its cells $Q^h_i$, $i=1,\dotsc,N=h^{-\spdim}$, we place a particle
$\vv{x}_i$ with weight $w_i:=h^\spdim$. The particles do not necessarily need
to be placed at the centres. The particles are then moved over time according
to $\dot{\vv{x}}_i(t)=\vv{a}\bigl(\vv{x}_i(t),t\bigr)$, $i=1,\dotsc,N$.
To fully specify the particle approximation $u_h(t) = \sum_{i=1}^{N}w_iu_i
\delta_{\vv{x}_i(t)}$, we define $u_i=u_0\bigl(\vv{x}_i(0)\bigr)$. 

\begin{lemma}\label{lem:quadrule}
Let $u_\sigma,v_\sigma\in V_\sigma^n(\Omega)$, $n\geq 2$, and $h\leq\sigma$. Then
for all times $t\in[0,T]$ one has with the hidden constant depending on $\vv{a}$,
$n$, $\spdim$, and $T$:
\begin{equation}
\label{eqn:quaderror}
\left|\idx{\Omega}{}{u_\sigma v_\sigma}{\vv{x}}-
\sum_{i=1}^{N}w_iu_\sigma\bigl(\vv{x}_i(t)\bigr)v_\sigma\bigl(\vv{x}_i(t)\bigr)\right|
\lesssim h\vert u_\sigma v_\sigma\vert_{W^{1,1}(\Omega)}.
\end{equation}
\end{lemma}
\begin{proof}
Let us abbreviate $f:=u_\sigma v_\sigma$, and note that because $n\geq 2$ we
have $f\in W^{1,\infty}(\Omega)$. We furthermore denote $f_0:=f\circ\traj{0}{t}$,
$f_{0,h}:=P_h^1f_0$, and note that the quadrature rule integrates $f_{0,h}\circ
\traj{t}{0}$ exactly. Because $\nabla\cdot\vv{a} = 0$ we have
$|{\kern-1pt}\det\boldnabla\traj{0}{t}|=|{\kern-1pt}\det\boldnabla\traj{t}{0}|
=1$.\autocite[Lemma~1.2]{raviart1985} We obtain by transforming the integral
and~\eqref{eqn:interpolanterror}:
\begin{equation}
\Vert f   - f_{0,h}\circ\traj{t}{0}\Vert_{L^1(\Omega)} =
\Vert f\circ\traj{0}{t} - f_{0,h}\Vert_{L^1(\Omega)} =
\Vert f_0 - P_h^1f_0\Vert_{L^1(\Omega)}
\stackrel{\eqref{eqn:interpolanterror}}{\lesssim} h
\vert f_0\vert_{W^{1,1}(\Omega)}.
\end{equation}
But because $\vv{a}$ is smooth and bounded, so are the derivatives of $\traj{0}{t}$
and we may write by Hölder's inequality $\vert f_0\vert_{W^{1,1}(\Omega)} = 
\vert f\circ\traj{0}{t}\vert_{W^{1,1}(\Omega)}\lesssim
\vert f\vert_{W^{1,1}(\Omega)}$. For the quadrature rule we obtain by the triangular
inequality:
\begin{equation}
\sum_{i=1}^{N}w_i\left(f-f_{0,h}\circ\traj{t}{0}\right)\bigl(\vv{x}_i(t)\bigr)
\leq
h^\spdim\sum_{i=1}^{N}\Vert f-f_{0,h}\circ\traj{t}{0}\Vert_{L^\infty(\traj{0}{t}(Q_i^h))}
= h^\spdim\sum_{i=1}^{N}\Vert f_0-f_{0,h}\Vert_{L^\infty(Q_i^h)}.
\end{equation}
We now make use of the properties of $P_h^1$ and the boundedness of $\boldnabla
\traj{0}{t}$ to obtain:
\begin{equation}
h^\spdim\sum_{i=1}^{N}\Vert f_0-f_{0,h}\Vert_{L^\infty(Q_i^h)}
\stackrel{\eqref{eqn:interpolanterror}}{\lesssim}
h^{1+\spdim}\sum_{i=1}^{N}\vert f_0\vert_{W^{1,\infty}(\hat{Q}_i^h)} \lesssim
h^{1+\spdim}\sum_{i=1}^{N}\vert f\vert_{W^{1,\infty}(\traj{0}{t}(\hat{Q}_i^h))}.
\end{equation}
Because $|\kern-1pt\det\boldnabla\traj{0}{t}|=1$ we have $\meas_{\spdim}(\traj{0}{t}(\hat{Q}^h_i))=
\meas_D(\hat{Q}^h_i)\lesssim h^\spdim$. Furthermore, by the Lipschitz continuity of
$\traj{0}{t}$, we obtain that $\diam(\traj{0}{t}(\hat{Q}_i^h))\lesssim h$. For
each index $i$ there therefore exists a bounded number of cubes $Q^\sigma_j$ whose
union $K_i:=\bigcup_j Q^\sigma_j$ covers $\traj{0}{t}(\hat{Q}^h_i)$ and that
fulfils $\meas_\spdim(K_i) \lesssim \sigma^\spdim$. The $K_i$ therefore cover
$\Omega$ about $(\tfrac{\sigma}{h})^\spdim$ times, and we thus obtain with together
with inverse estimate~\eqref{eqn:globinverseineq}:
\begin{multline}
h^{1+\spdim}\sum_{i=1}^{N}\vert f\vert_{W^{1,\infty}(\traj{0}{t}(\hat{Q}_i^h))}\leq
h^{1+\spdim}\sum_{i=1}^{N}\vert f\vert_{W^{1,\infty}(K_i)}
\stackrel{\eqref{eqn:globinverseineq}}{\lesssim}
h\left(\frac{h}{\sigma}\right)^\spdim \sum_{i=1}^{N}\vert f\vert_{W^{1,1}(K_i)} \\\lesssim
h\left(\frac{h}{\sigma}\right)^\spdim \left(\frac{\sigma}{h}\right)^\spdim
\vert f\vert_{W^{1,1}(\Omega)} = h\vert f\vert_{W^{1,1}(\Omega)}.
\end{multline}
Thus we obtain in total:
\begin{multline}
\left|\idx{\Omega}{}{u_\sigma v_\sigma}{\vv{x}}-
\sum_{i=1}^{N}w_iu_\sigma\bigl(\vv{x}_i(t)\bigr)v_\sigma\bigl(\vv{x}_i(t)\bigr)\right|
= \\
\left|\idx{\Omega}{}{f-f_{0,h}\circ\traj{t}{0}}{\vv{x}}-
\sum_{i=1}^{N}w_i\bigl(f-f_{0,h}\circ\traj{t}{0}\bigr)\bigl(\vv{x}_i(t)\bigr)\right|
\lesssim h\vert u_\sigma v_\sigma\vert_{W^{1,1}(\Omega)}.
\end{multline}
\end{proof}

\subsection{Convergence}
Let $n\in\mathbb{N}$, $n\geq 2$, and $1\leq p\leq\infty$. Given the quadrature
rule from \cref{sec:particleapprox} at some time $t\in[0,T]$, we define the
following operators:
\begin{align}
A: V_\sigma^{n,p}(\Omega)\to V_\sigma^{-n,p}(\Omega),\quad
\langle Au_\sigma, v_\sigma\rangle &:= \idx{\Omega}{}{u_\sigma v_\sigma}{\vv{x}}, \\
A_h: V_\sigma^{n,p}(\Omega)\to V_\sigma^{-n,p}(\Omega),\quad
\langle A_hu_\sigma, v_\sigma\rangle &:= \sum_{i=1}^{N}w_iu_\sigma\bigl(\vv{x}_i(t)\bigr) v_\sigma\bigl(\vv{x}_i(t)\bigr).
\end{align}

\begin{lemma}[Stability]
Let $h=d\sigma$, with $0<d<1$ independent of $\sigma$ but small enough and
$n\geq 2$. Then for all $t\in[0,T]$ the operator $A_h$ is invertible and for all
$1\leq p\leq\infty$ its inverse is bounded:
\begin{equation}\label{eqn:stability}
\Vert A_h^{-1}\Vert_{V_\sigma^{-n,p}(\Omega)\to V_\sigma^{n,p}(\Omega)}\leq C.
\end{equation}
\end{lemma}
\begin{proof}
The proof for $p=2$ is illustrative. Using the quadrature error estimate
\Cref{eqn:quaderror} and inverse inequality~\eqref{eqn:globinverseineq}, one
obtains that the operator $A_h$ is coercive, i.\,e., for all $v_\sigma\in
V_\sigma^n(\Omega)$ has:
\begin{multline}
\langle A_h v_\sigma,v_\sigma\rangle = \langle A v_\sigma,v_\sigma\rangle -
\langle (A-A_h)v_\sigma,v_\sigma\rangle 
\stackrel{\eqref{eqn:quaderror}}{\geq}
\Vert v_\sigma\Vert_{L^2(\Omega)}^2 - Ch|v_\sigma^2|_{W^{1,1}(\Omega)} \\
\stackrel{\eqref{eqn:globinverseineq}}{\geq}
\left(1-\tilde{C}\frac{h}{\sigma}\right)\Vert v_\sigma\Vert_{L^2(\Omega)}^2 =
(1-\tilde{C}d)\Vert v_\sigma\Vert_{L^2(\Omega)}^2.
\end{multline}
For small enough $d$ the operator $A_h$ thus has an inverse that is bounded:
$\Vert A_h^{-1}\Vert_{V_\sigma^{-n,2}(\Omega)\to V_\sigma^{n,2}(\Omega)}\leq C$.
Note that $\tilde{C}$ neither depends on $h$ nor on $\sigma$, i.\,e., $d$ can
also be chosen independent of $h$ and $\sigma$.

The exact operator $A^{-1}$ is the $L^2(\Omega)$-projector onto $V_\sigma^n
(\Omega)$. Its boundedness for $1\leq p\leq\infty$ has been shown by Douglas,
Dupont, and Wahlbin\autocite{douglas1974}, as well as Crouzeix and Thomée.%
\autocite{crouzeix1987} The proof is technical, but with only minor modifications
to account for the quadrature error directly carries over to $A_h^{-1}$. These
modifications can be found in the appendix. Thus
$\Vert A_h^{-1}\Vert_{V_\sigma^{-n,p}(\Omega)\to V_\sigma^{n,p}(\Omega)}\leq C$,
$1\leq p\leq\infty$.
\end{proof}

As a direct corollary, one obtains that the system matrix corresponding
to $A_h$ in terms of the B-spline basis is not only sparse, but also
symmetric positive definite and well-conditioned. Using the conjugate
gradient method, the solution of~\eqref{eqn:neqapproach} can thus be
computed approximately at optimal time and space complexity $\bigO{N}$.
This stability result will allow us to establish convergence. Because
we take point evaluations of the initial data $u_0$, it is most natural
to consider the case $p=\infty$.

\begin{theorem}[$L^\infty(\Omega)$-Convergence]\label{thm:linfconvergence}
Let $u(\cdot,t)$, $t\in[0,T]$, denote the solution of the advection equation
\eqref{eqn:advequation} with continuous initial data $u_0$. Let $u_h(t)$ denote
the particle approximation from \cref{sec:particleapprox}. Then for $h=d\sigma$,
with $0<d<1$ independent of $\sigma$ but small enough, and $n\geq 2$ the
following error-bound holds for the solution of~\eqref{eqn:neqapproach}, i.\,e.,
for $u_{h,\sigma}(\cdot,t):=A_h^{-1}u_h$:
\begin{equation}
\Vert u_{h,\sigma}(\cdot,t) -  u(\cdot,t)\Vert_{L^\infty(\Omega)} \lesssim
\sigma^s\Vert u_0\Vert_{W^{s,\infty}(\Omega)}\qquad 0\leq s\leq n.
\end{equation}
\end{theorem}
\begin{proof}
The key observation is the following. Let us for the moment assume that at the
current time~$t$ the exact solution $u(\cdot,t)$ of the advection equation would
be a spline: $u(\cdot,t)\in V_\sigma^n(\Omega)$. Noting that $u(\vv{x}_i(t),t) =
u_i$ exactly, we immediately see that $u_{h,\sigma}:= u(\cdot,t)$ solves
\eqref{eqn:neqapproach}. Because $A_h$ is invertible, this also is the only
solution, and in this case our approach yields the exact result.

In the general case we let $u_\sigma:=P_\sigma^n u(\cdot,t)$ denote the quasi-%
interpolant of $u$ onto $V_\sigma^n(\Omega)$. Furthermore, we set $(u_\sigma)_h
:=\sum_{i=1}^{N}w_iu_\sigma\bigl(\vv{x}_i(t)\bigr)\delta_{\vv{x}_i(t)}$ and
obtain:
\begin{multline}
\Vert u_{h,\sigma}-u(\cdot,t)\Vert_{L^\infty(\Omega)} = 
\Vert A_h^{-1}u_h - u\Vert_{L^\infty(\Omega)}\\\leq 
\Vert A_h^{-1}\bigl(u_h - (u_\sigma)_h\bigr)\Vert_{L^\infty(\Omega)} +
\underbrace{\Vert A_h^{-1}(u_\sigma)_h - u_\sigma\Vert_{L^\infty(\Omega)}}_{=0} +
\Vert u_\sigma - u(\cdot,t)\Vert_{L^\infty(\Omega)} \\
\stackrel{\eqref{eqn:stability}}{\lesssim}
\Vert u_h - (u_\sigma)_h\Vert_{V_{\sigma}^{-n,\infty}(\Omega)} +
\Vert u_\sigma - u(\cdot,t)\Vert_{L^\infty(\Omega)}.
\end{multline}
By \eqref{eqn:interpolanterror}, the last term can be bounded by $\sigma^s\Vert
u(\cdot,t)\Vert_{W^{s,\infty}(\Omega)}$, and the stability of the advection
equation~\eqref{eqn:advstability} furthermore yields $\Vert u(\cdot,t)\Vert_{W^{s,\infty}(\Omega)}
\lesssim \Vert u_0\Vert_{W^{s,\infty}(\Omega)}$. 

It remains to show that we also have $\Vert u_h-(u_\sigma)_h\Vert_{V_{\sigma}^%
{-n,\infty}(\Omega)}\lesssim\sigma^s\Vert u_0\Vert_{W^{s,\infty}(\Omega)}$. Thus
let $v_\sigma\in V_\sigma^n(\Omega)$ be arbitrary but fixed. Because we have
$u_i = u\bigl(\vv{x}_i(t),t\bigr)$ exactly, one obtains:
\begin{equation}
\sum_{i=1}^{N}w_i\left(u_i-u_\sigma\bigl(\vv{x}_i(t)\bigr)\right)v_\sigma\bigl(\vv{x}_i(t)\bigr)
\leq
\Vert u(\cdot,t)-u_\sigma\Vert_{L^\infty(\Omega)}\,\cdot\,%
\sum_{i=1}^{N}w_i\Vert v_\sigma\Vert_{L^\infty(\traj{0}{t}(Q_i^h))}.
\end{equation}
The sum can be bounded by $C\Vert v_\sigma\Vert_{L^1(\Omega)}$ using the same
techniques as in the proof of \Cref{lem:quadrule}, while for the first part
we already established $\Vert u(\cdot,t)-u_\sigma\Vert_{L^\infty(\Omega)}
\lesssim\sigma^s\Vert u_0\Vert_{W^{s,\infty}(\Omega)}$, and thus we in fact have
$\Vert u_h-(u_\sigma)_h\Vert_{V_{\sigma}^{-n,\infty}(\Omega)}\lesssim\sigma^s
\Vert u_0\Vert_{W^{s,\infty}(\Omega)}$ as desired.
\end{proof}

In order to show convergence in $L^p(\Omega)$ for $p\neq\infty$, one requires
inverse estimates. Let us thus consider the case where one initialises the
particle approximation using $P_\sigma^{n+1}u_0$ instead of $u_0$ itself. For the
initialisation we need to chose order $n+1$ as opposed to just $n$ to ensure that
we have $P_\sigma^{n+1}u_0\in W^{n,\infty}(\Omega)$.
\begin{theorem}[$L^p(\Omega)$-Convergence]
Let $u(\cdot,t)$, $t\in[0,T]$, denote the solution of the advection equation
\eqref{eqn:advequation} with initial data $u_0\in L^p(\Omega)$, $1\leq p\leq
\infty$. Let $u_h(t)$ denote the particle approximation from
\cref{sec:particleapprox}, but with $u_i := \bigl(P_\sigma^{n+1}u_0\bigr)
(\vv{x}_i(0))$, $i=1,\dotsc,N$, $n\geq 2$. Then for $h=d\sigma$, with $0<d<1$
independent of $\sigma$ but small enough the following error-bound holds for the
solution of~\eqref{eqn:neqapproach}, i.\,e., for $u_{h,\sigma}(\cdot,t):=
A_h^{-1}u_h$:
\begin{equation}
\Vert u_{h,\sigma}(\cdot,t) -  u(\cdot,t)\Vert_{L^p(\Omega)} \lesssim
\sigma^s\Vert u_0\Vert_{W^{s,p}(\Omega)}\qquad 0\leq s\leq n,\ p\in[1,\infty].
\end{equation}
\end{theorem}
\begin{proof}
Let $\tilde{u}(\cdot,t)$ denote the solution of the advection equation with
$u_0$ replaced by $P_\sigma^{n+1}u_0$. We then have:
\begin{equation}
\Vert u(\cdot,t) - \tilde{u}(\cdot,t)\Vert_{L^p(\Omega)}
\stackrel{\eqref{eqn:advstability}}{\lesssim}
\Vert u_0 - P_\sigma^{n+1}u_0\Vert_{L^p(\Omega)} 
\stackrel{\eqref{eqn:interpolanterror}}{\lesssim} \sigma^s\Vert u_0\Vert_{W^{s,p}(\Omega)}
\qquad 0\leq s\leq n+1.
\end{equation}
Therefore, the convergence order of the method does not deteriorate by replacing the
initial data with its spline approximation. In complete analogy to the proof of
\Cref{thm:linfconvergence}, one obtains:
\begin{equation}
\Vert u_{h,\sigma}-\tilde{u}(\cdot,t)\Vert_{L^p(\Omega)} \lesssim
\Vert u_h-(u_\sigma)_h\Vert_{V_\sigma^{-n,p}(\Omega)} +
\Vert u_\sigma - \tilde{u}(\cdot,t)\Vert_{L^p(\Omega)},
\end{equation}
where the last term may be again bounded as desired:
\begin{equation}
\Vert u_\sigma - \tilde{u}(\cdot,t)\Vert_{L^p(\Omega)}
\stackrel{\eqref{eqn:interpolanterror}}{\lesssim}
\sigma^s\Vert\tilde{u}(\cdot,t)\Vert_{W^{s,p}(\Omega)}
\stackrel{\eqref{eqn:advstability}}{\lesssim}
\sigma^s\Vert P_\sigma^{n+1}u_0\Vert_{W^{s,p}(\Omega)}
\stackrel{\eqref{eqn:interpolantbound}}{\lesssim}
\sigma^s\Vert u_0\Vert_{W^{s,p}(\Omega)},\qquad  0\leq s\leq n.
\end{equation}
For the remaining term one obtains for arbitrary $v_\sigma\in V_\sigma^n(\Omega)$
using Hölder's inequality for $\tfrac{1}{p}+\tfrac{1}{q}=1$ and the usual
modifications for $p=\infty$ or $q=\infty$:
\begin{equation}
\sum_{i=1}^{N}w_i\left(u_i-u_\sigma\bigl(\vv{x}_i(t)\bigr)\right)v_\sigma\bigl(\vv{x}_i(t)\bigr)
\leq \left(\sum_{i=1}^{N}w_i\Vert\tilde{u}(\cdot,t)-u_\sigma\Vert_{L^\infty(\traj{0}{t}(Q_i^h))}^p\right)^{\frac{1}{p}}
\left(\sum_{i=1}^{N}w_i\Vert v_\sigma\Vert_{L^\infty(\traj{0}{t}(Q_i^h))}^q\right)^{\frac{1}{q}}.
\end{equation}
Now, using $K_i$ from the proof of~\Cref{lem:quadrule}, the smoothness and
boundedness of $\traj{0}{t}$ and its inverse $\traj{t}{0}$, the first sum may be
bounded as follows:
\begin{multline}
\left(\sum_{i=1}^{N}w_i\Vert\tilde{u}(\cdot,t)-u_\sigma\Vert_{L^\infty(\traj{0}{t}(Q_i^h))}^p\right)^{\frac{1}{p}}
\leq
h^\frac{\spdim}{p}\left(\sum_{i=1}^{N}\Vert\tilde{u}(\cdot,t)-u_\sigma\Vert_{L^\infty(K_i)}^p\right)^{\frac{1}{p}}\\
\stackrel{\eqref{eqn:interpolanterror}}{\lesssim}
h^\frac{\spdim}{p}\sigma^s \left(\sum_{i=1}^{N}\Vert\tilde{u}(\cdot,t)\Vert_{W^{s,\infty}(\hat{K}_i)}^p\right)^{\frac{1}{p}}
\lesssim
h^\frac{\spdim}{p}\sigma^s \left(\sum_{i=1}^{N}\Vert P_\sigma^{n+1}u_0\Vert_{W^{s,\infty}(\traj{t}{0}(\hat{K}_i))}^p\right)^{\frac{1}{p}}.
\end{multline}
Analogously to the $K_i$, we may find a bounded number of cubes $Q_i^\sigma$ that
covers $\traj{t}{0}(\hat{K}_i)$, such that their union $L_i$ fulfils $\meas_\spdim(L_i)
\lesssim\sigma^\spdim$. Moreover, the $L_i$ cover $\Omega$ about
$\left(\tfrac{\sigma}{h}\right)^\spdim$ times. Thus we obtain using inverse
estimate~\eqref{eqn:globinverseineq}:
\begin{multline}
h^\frac{\spdim}{p}\sigma^s \left(\sum_{i=1}^{N}\Vert P_\sigma^{n+1}u_0\Vert_{W^{s,\infty}(\traj{t}{0}(\hat{K}_i))}^p\right)^{\frac{1}{p}}
\leq
h^\frac{\spdim}{p}\sigma^s \left(\sum_{i=1}^{N}\Vert P_\sigma^{n+1}u_0\Vert_{W^{s,\infty}(L_i)}^p\right)^{\frac{1}{p}} \\
\stackrel{\eqref{eqn:globinverseineq}}{\lesssim}
\left(\frac{h}{\sigma}\right)^\frac{\spdim}{p}\sigma^s
\left(\sum_{i=1}^{N}\Vert P_\sigma^{n+1}u_0\Vert_{W^{s,p}(L_i)}^p\right)^{\frac{1}{p}}
\lesssim
\left(\frac{h}{\sigma}\right)^\frac{\spdim}{p}\left(\frac{\sigma}{h}\right)^\frac{\spdim}{p}\sigma^s
\Vert P_\sigma^{n+1}u_0\Vert_{W^{s,p}(\Omega)}
\stackrel{\eqref{eqn:interpolantbound}}{\lesssim}
\sigma^s\Vert u_0\Vert_{W^{s,p}(\Omega)}.
\end{multline}
Similarly, we obtain:
\begin{equation}
\left(\sum_{i=1}^{N}w_i\Vert v_\sigma\Vert_{L^\infty(\traj{0}{t}(Q_i^h))}^q\right)^{\frac{1}{q}}
\lesssim \Vert v_\sigma\Vert_{L^q(\Omega)},
\end{equation}
and thus $\Vert u_h-(u_\sigma)_h\Vert_{V_\sigma^{-n,p}(\Omega)}\lesssim
\sigma^s\Vert u_0\Vert_{W^{s,p}(\Omega)}$.
\end{proof}

\subsection{Relation to Blob-Methods}
The regularisation by projection approach may be called a Particle-in-Cell
scheme, because of the presence of an underlying grid. However, the approach
also corresponds to a classic blob-based method with a specially chosen
blob-function $\zeta_\sigma(\vv{x},\vv{y})$. In fact, for each $\vv{y}\in\Omega$,
let us define the functions $\zeta_\sigma(\cdot,\vv{y})$ as $A^{-1}\delta_\vv{y}$
and $\zeta_{h,\sigma}(\cdot,\vv{y}):=A_h^{-1}\delta_\vv{y}$. We then have
$A^{-1}u_h =\sum_{i=1}^{N}w_iu_i\zeta_\sigma(\vv{x},\vv{x}_i)$ and $A_h^{-1}u_h =
\sum_{i=1}^{N}w_iu_i\zeta_{h,\sigma}(\vv{x},\vv{x}_i)$, respectively. 

The projection approaches are thus in fact blob-based methods. Both
$\zeta_\sigma(\cdot,\vv{y})$ and $\zeta_{h,\sigma}(\cdot,\vv{y})$ are 
decaying at an exponential rate away from $\vv{y}$, just like conventional
blob-functions. Moreover, $\zeta_\sigma$ fulfils the moment conditions:
\begin{equation}
\idx{\Omega}{}{\zeta_\sigma(\vv{x},\vv{y})\vv{x}^\alpha}{\vv{x}} =
\langle AA^{-1}\delta_\vv{y},\vv{x}^\alpha\rangle =
\vv{x}^\alpha|_{\vv{x}=\vv{y}}\qquad \forall|\alpha| < n,
\end{equation}
while $\zeta_{h,\sigma}$ fulfils \emph{discrete moment conditions}:
\begin{equation}
\sum_{i=1}^{N}w_i\zeta_{h,\sigma}(\vv{x}_i,\vv{y})\vv{x}_i^\alpha =
\langle A_hA_h^{-1}\delta_\vv{y},\vv{x}^\alpha\rangle =
\vv{x}^\alpha|_{\vv{x}=\vv{y}}\qquad \forall|\alpha|< n.
\end{equation}

The fact that these special blob-functions exist means that other techniques
developed for blob-based approaches can also be applied in the current setting.
Let us for example consider viscous effects with viscosity $\nu>0$. Interestingly,
for $\zeta_{h,\sigma}$, both Fishelov's scheme\autocite{fishelov1990} and the
method of particle strength exchange\autocite{degond1989} coincide and reduce to:
\begin{equation}
\left\lbrace
\begin{aligned}
\ddx{\vv{x}_i}{t}(t) &= \vv{a}\bigl(\vv{x}_i(t),t\bigr),\\
\ddx{u_i}{t}(t) &= \nu\Delta u_{h,\sigma}\bigl(\vv{x}_i(t)\bigr),
\end{aligned}
\right.\qquad i=1,\dotsc,N.
\end{equation}

\section{Numerical Experiments}\label{sec:experiments}
In this section we will consider four different types of numerical experiments
with (vortex) particle methods. In the first example we consider a low order
computation on a two-dimensional benchmark that has been used before to visualise
the necessity of remeshing in classical particle methods.\autocite{koumoutsakos2008}
Our experiment will show that by using $A_h^{-1}$ no remeshing is necessary.

The second experiment similarly is of graphical nature: we apply the particle
method to the problem of interface tracking using a level-set function. This
function evolves over time according to the linear advection equation and thus
perfectly fits into the framework considered here. The results highlight the
absence of numerical diffusion in the scheme.

For the third series of experiments we developed a solver for the two-dimensional
domain $\Omega = (0,1)^2$ with periodic boundary conditions. We perform high-%
order, long-term simulations of a quasi-steady, but highly instable flow. Due to
its instability, this flow is notoriously hard to accurately reproduce in long-%
term simulations. We compare the vortex method to a state-of-the art flow solver:%
\autocite{schroeder2018} an eighth order, exactly divergence-free, hybridised
discontinuous Galerkin (HDG8) method. The results show the vortex method to be
competitive.

Finally, the fourth series of experiments is a convergence study on a fully three-%
dimensional flow-problem: the Arnold--Beltrami--Childress (ABC) flow. Despite the
additional vortex stretching term in three-dimensional space, the vortex method
remained stable.

\subsection{Graphical Demonstration}
Koumoutsakos, Cottet, and Rossinelli\autocite{koumoutsakos2008} describe the
following benchmark case in two dimensions in order to illustrate the necessity
of remeshing. Let us consider the two-dimensional, incompressible Euler
equations in their vorticity formulation in the whole-space:
\begin{equation}
\pdx{\omega}{t} + (\vv{u}\cdot\nabla)\omega = 0.
\end{equation}
Here the advected quantity is the vorticity. Following the fluid mechanics
convention, it is labelled  $\omega$ instead of $u$, while the velocity is
denoted by $\vv{u}$ instead of $\vv{a}$. It is computed from $\omega$ via:
\begin{align}\label{eqn:streamfct2d}
\psi(\vv{x}) &:= G\star\omega := \frac{1}{2\pi}\idx{\reals^2}{}{%
\ln\left(\frac{1}{|\vv{x}-\vv{y}|}\right)\omega(\vv{y})}{\vv{y}}, \\
\vv{u}\bigl(\vv{x}) &:= \mathrm{curl}\,\psi :=
\left(\pdx{\psi}{x_2}(\vv{x}),-\pdx{\psi}{x_1}(\vv{x})\right)^\top.
\end{align}
A steady solution of this equation is given by:
\begin{equation}
\omega(\vv{x},t) = 100\,\max\lbrace1-2|\vv{x}|,0\rbrace.
\end{equation}
Note that the vorticity is compactly supported, while the velocity has global
support. The streamlines corresponding to $\vv{u}$ are concentric circles around
the origin.

The task is now to construct a vortex particle method that reproduces this
result over extended periods of time. We will artificially restrict ourselves
to the domain $\Omega:=(-1,1)^2$, which contains the entire support of $\omega$.
Due to the circular motion particles will inevitably leave this domain. For this
reason our particle field $\omega_h(0)$ will be created and tracked on the
slightly larger domain $\Xi:=(-2,2)^2$ as described in \cref{sec:particleapprox}.
We will use $\sigma = 0.01$, $d=0.5$, $h=0.005$. At any time $t$, only the particles
that are currently located inside of $\Omega$ will be considered for computing
$\omega_{h,\sigma} := A_h^{-1}\omega_h\in V_\sigma^n(\Omega)$. We choose $n=2$:
it does not make sense to chose higher orders due to the low regularity of the
exact solution.

In order to compute the velocity field, $\omega_{h,\sigma}$ is extended with
zero outside of $\Omega$ and inserted into equation \eqref{eqn:streamfct2d}.
The integral can be evaluated analytically: $\omega_{h,\sigma}$ is a piece-wise
polynomial on a Cartesian grid. For a faster evaluation, however, we instead
compute $\psi_{h,\sigma}$, the $L^2(\Xi)$-projection of $G\star\omega_{h,\sigma}$
onto $V_\sigma^{n+2}(\Xi)$. The computation of this projection can be accelerated
by a fast multipole method. The resulting function $\vv{u}_{h,\sigma} :=
\mathrm{curl}\,\psi_{h,\sigma}$ is supported on all of $\Xi$ and by construction
exactly divergence-free. It is used to convect \emph{all particles}---also
those outside of $\Omega$---according to $\dot{\vv{x}}_i(t) =
\vv{u}_{h,\sigma}\bigl(\vv{x}_i(t)\bigr)$. These ODEs are discretised using the
classical Runge--Kutta method and a fixed time-step of $\Delta t = 0.005$.

\begin{figure}
\centering
\begin{subfigure}{0.3\textwidth}
\includegraphics[width=\textwidth]{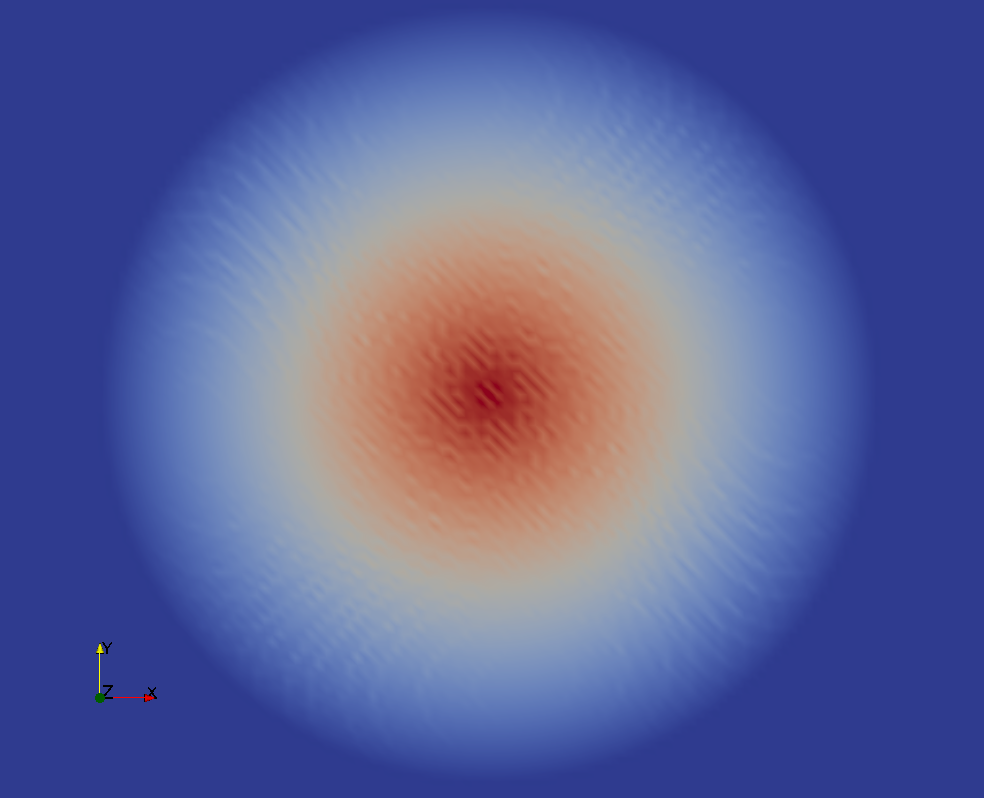}
\end{subfigure}
\hfil %
\begin{subfigure}{0.3\textwidth}
\includegraphics[width=\textwidth]{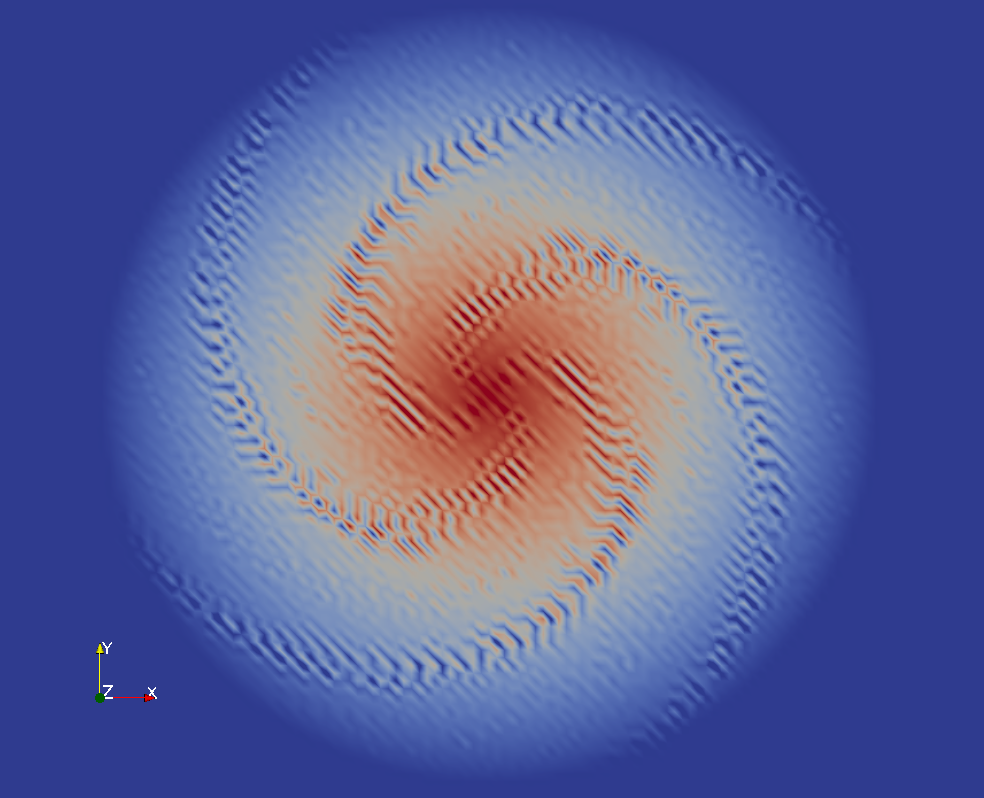}
\end{subfigure}
\hfil %
\begin{subfigure}{0.3\textwidth}
\includegraphics[width=\textwidth]{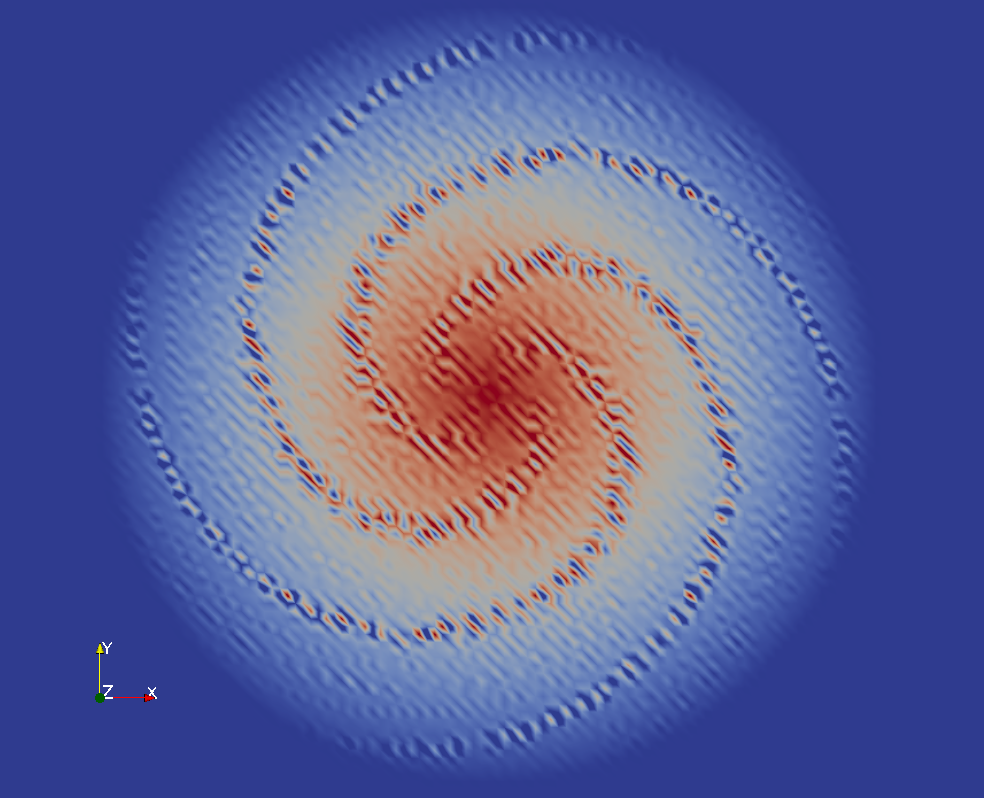}
\end{subfigure}
\\[0.5cm]
\begin{subfigure}{0.3\textwidth}
\includegraphics[width=\textwidth]{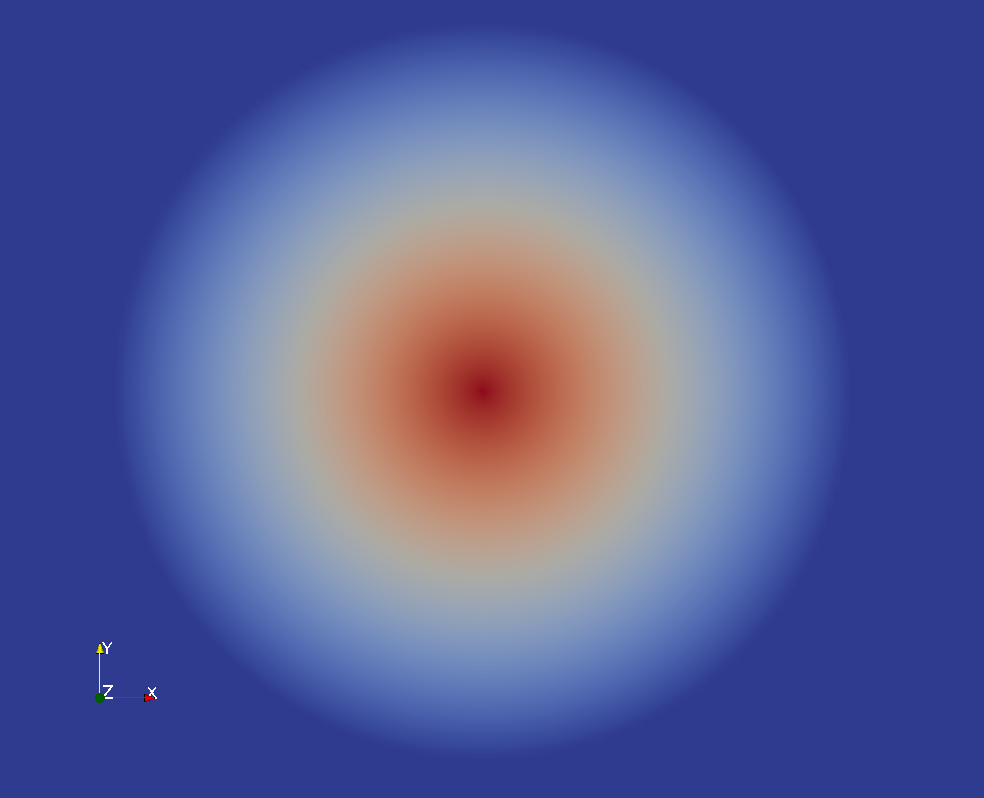}
\subcaption*{$t=0.01$}
\end{subfigure}
\hfil %
\begin{subfigure}{0.3\textwidth}
\includegraphics[width=\textwidth]{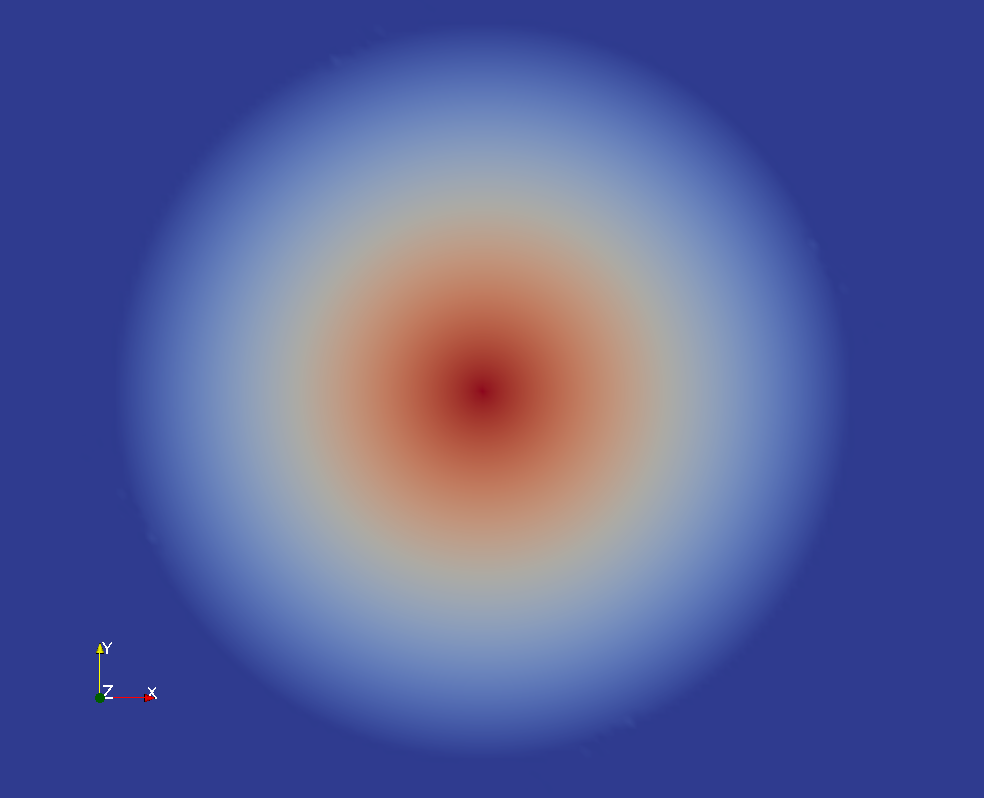}
\subcaption*{$t=0.10$}
\end{subfigure}
\hfil %
\begin{subfigure}{0.3\textwidth}
\includegraphics[width=\textwidth]{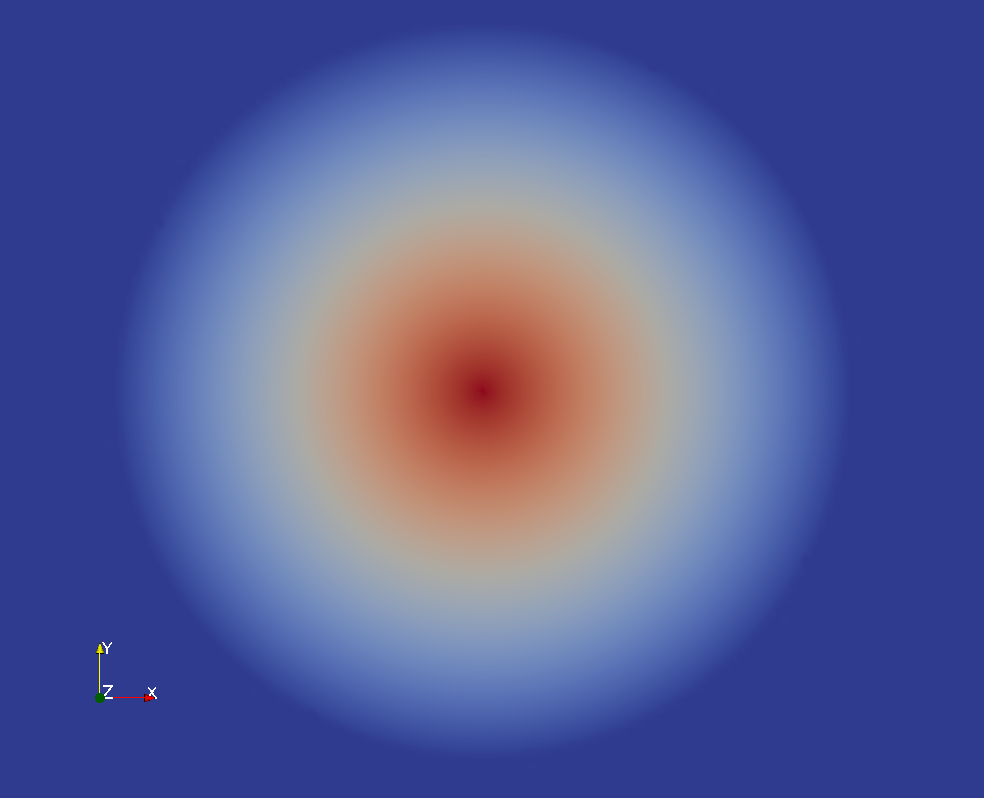}
\subcaption*{$t=0.15$}
\end{subfigure}
\caption{\label{fig:koumoutsakos-bench}A comparison of $A^{-1}\omega_h$ in
the first row and $A^{-1}_h\omega_h$ in the second. The first row corresponds
to a classic vortex blob-method. Artefacts already appear in the beginning of
the simulation and worsen over time. The second row shows the results of the
new scheme which remain virtually unchanged over time as desired.}
\end{figure}

The results of this experiment are depicted in \Cref{fig:koumoutsakos-bench}.
In the first row the simulation was carried out with $A^{-1}u_h$ instead of
$\omega_{h,\sigma}=A^{-1}_h\omega_h$. The first row clearly shows that the error
quickly increases with time $t$. The reader is invited to compare this picture
with those of Koumoutsakos, Cottet, and Rossinelli\autocite{koumoutsakos2008}.
On the other hand, the solution using $A_h^{-1}\omega_h$ remained accurate,
despite the complete absence of any remeshing.

\subsection{Application to Zalesak's Disk}
A great advantage of particle methods when applied to the advection equation is
their complete lack of numerical diffusion. This makes them particularly interesting
for interface tracking using the level-set method. A long established benchmark
problem in the field is Zalesak's disk, in which the evolution of a slitted disk
subject to a rigid body rotation is tracked over time.~\autocite{zalesak1979}
The core difficulty here is to maintain the sharp kinks and corners of this domain:
in many conventional schemes the corners quickly get smeared out.

We consider Zalesak's disk on the domain $\Omega:=[-0.5,0.5]^2$ and the time
interval $t\in[0,628]$. The quantity of interest here is the signed distance
function $u$, whose initial data $u_0$ is given as an algorithm in the appendix.
This function evolves over time according to the advection equation~\eqref{eqn:advequation}
with a given velocity field. In this test-case we thus follow the notation of
\Cref{sec:particlemethods} and give the velocity field as:
\begin{equation}
\vv{a}(x,y) := \frac{2\pi}{628}\begin{pmatrix}-y\\x\end{pmatrix}.
\end{equation}

The discretisation is analogous to the previous subsection, with regularisation
taking place on $\Omega$, but particles being tracked on $\Xi := [-1,1]^2 \supset
\Omega$, with $n=2$, $\sigma = 0.01$, $d=\tfrac{1}{2}$, $h=0.005$, $\Delta t = 1$,
where the classical Runge--Kutta method is used as a time-stepping scheme. The
results together with the contour line of $u=0$ are depicted in \Cref{fig:zalesak}.
It can clearly be seen that the interface remains well-maintained, and does not
degenerate over time. Due to the Lagrangian nature of the method, the results for
$t=0$ and $t=628$ are identical, so this experiment can be extended to arbitrary
time intervals.

\begin{figure}
\centering
\begin{tabular}{ccc}
\includegraphics[width=0.3\textwidth]{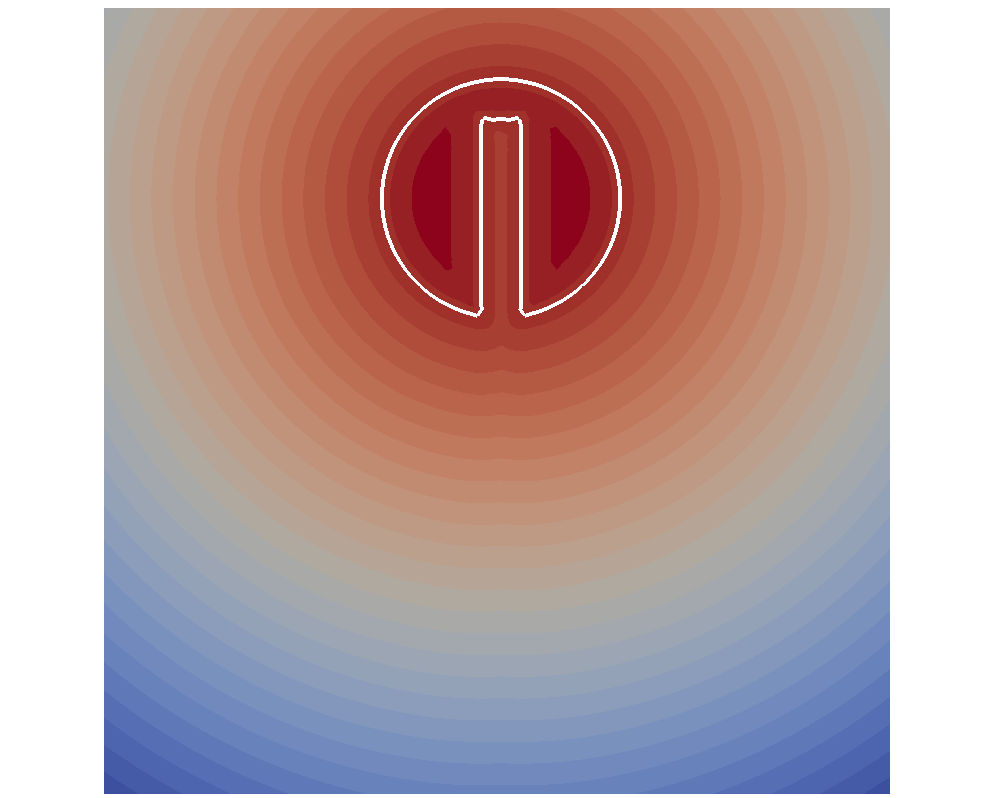} &
\includegraphics[width=0.3\textwidth]{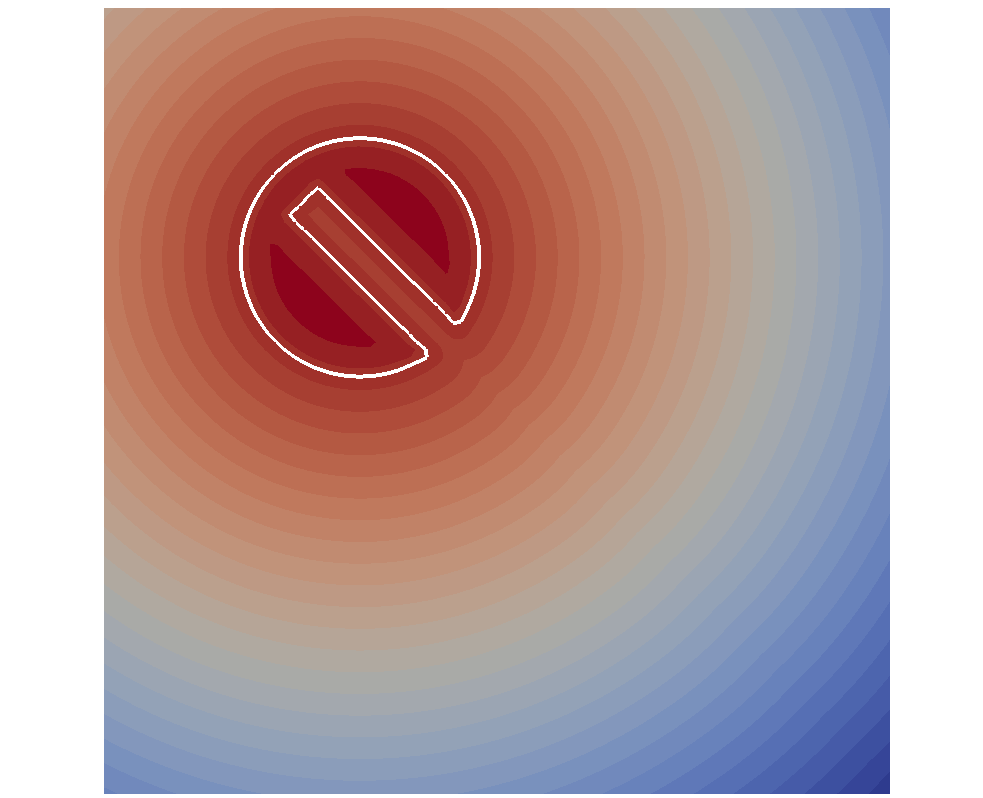} &
\includegraphics[width=0.3\textwidth]{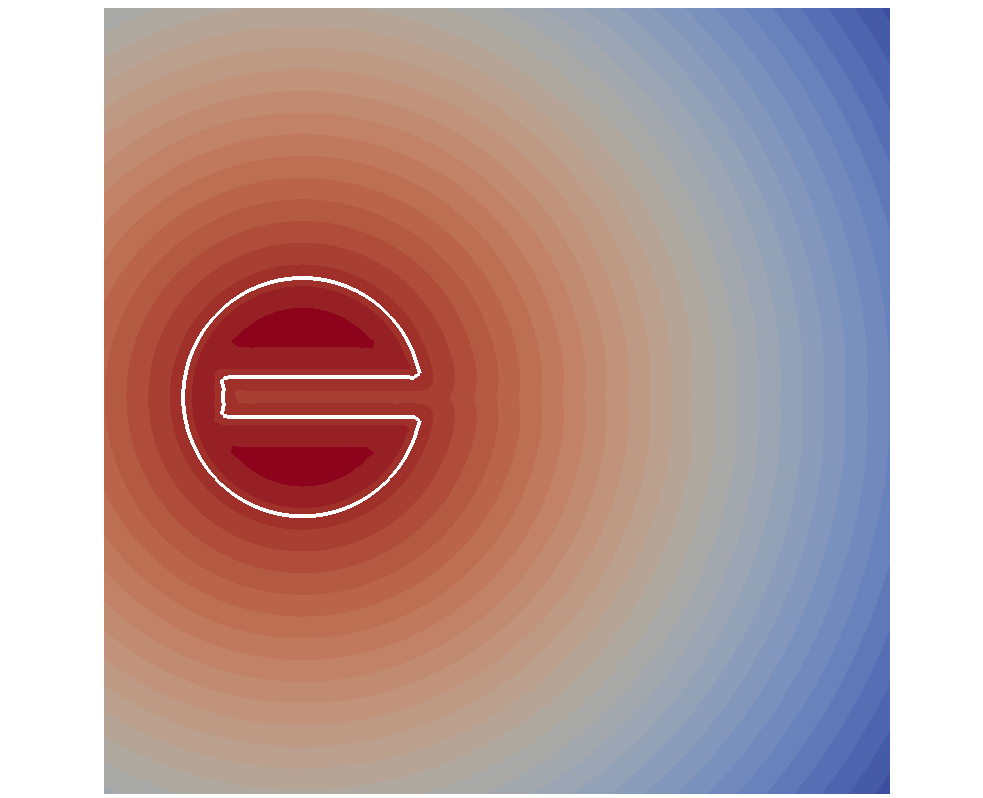} \\[0.5cm]
\includegraphics[width=0.3\textwidth]{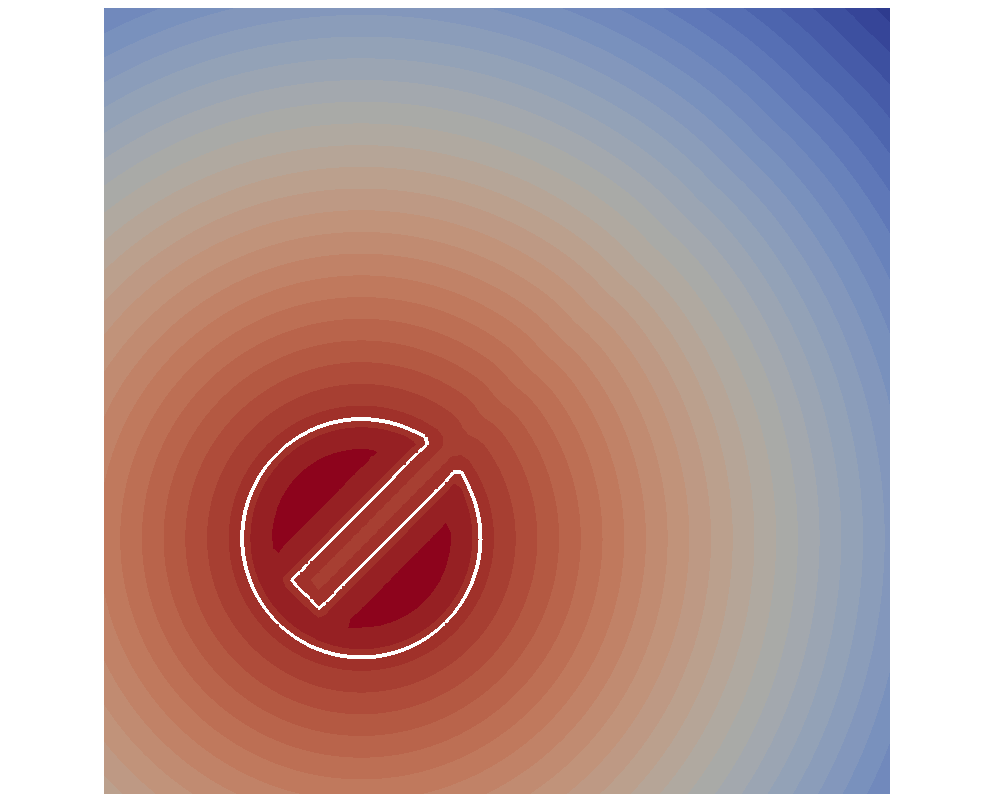} &
\includegraphics[width=0.3\textwidth]{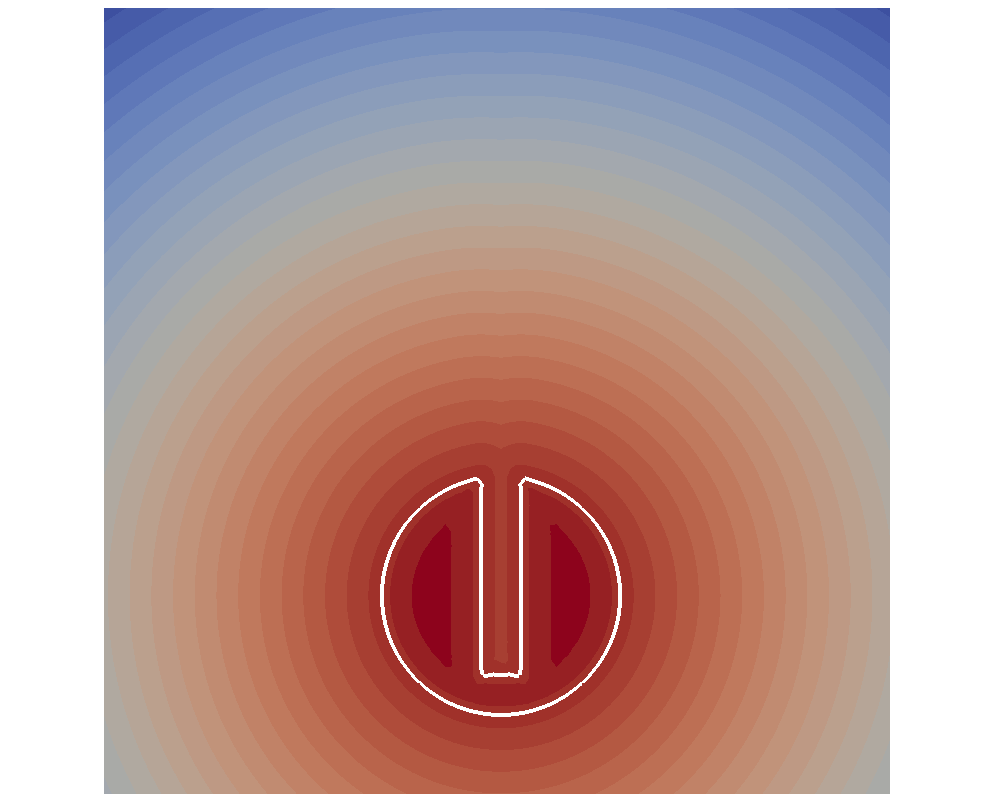} &
\includegraphics[width=0.3\textwidth]{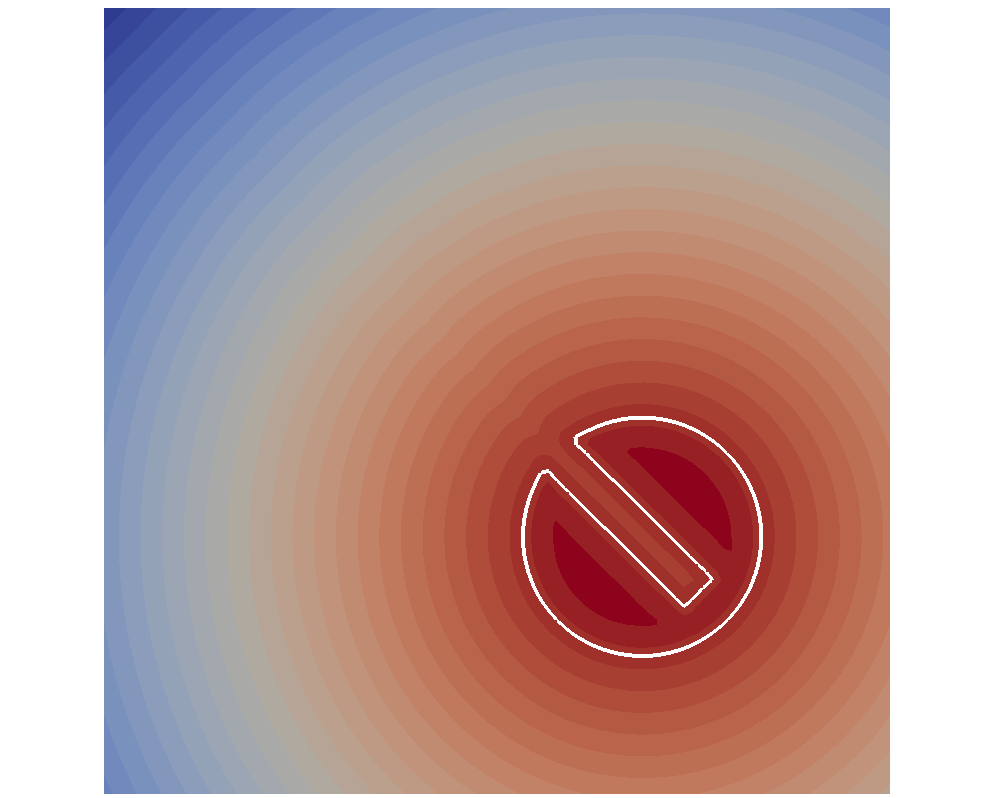} \\[0.5cm]
\includegraphics[width=0.3\textwidth]{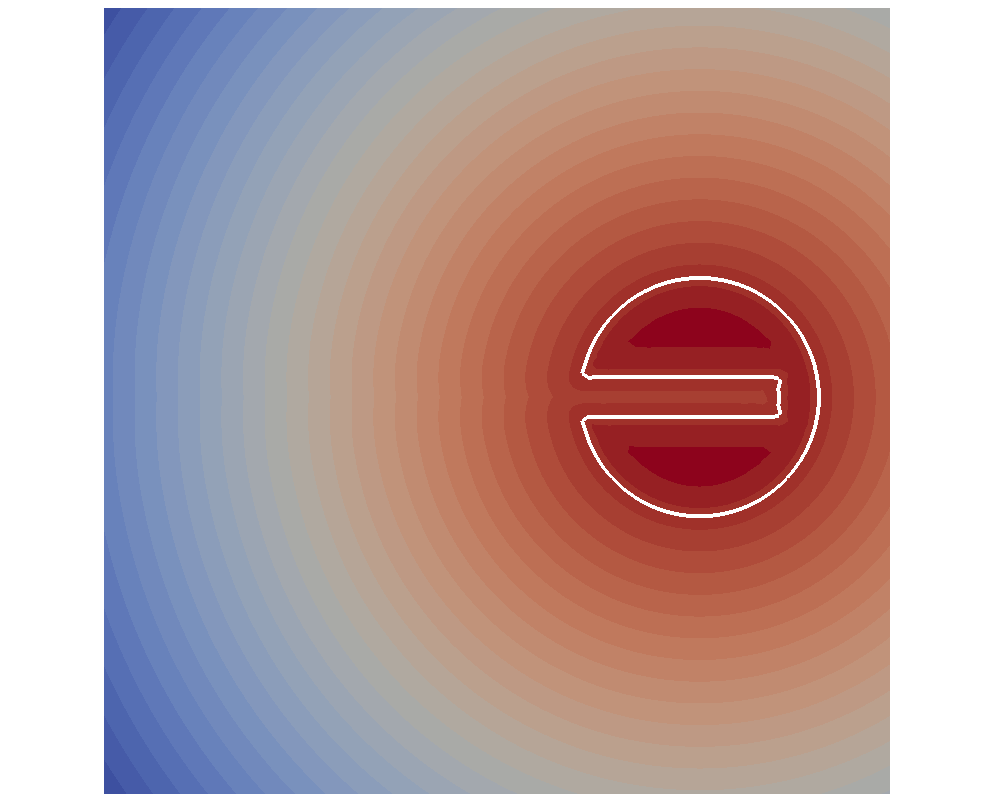} &
\includegraphics[width=0.3\textwidth]{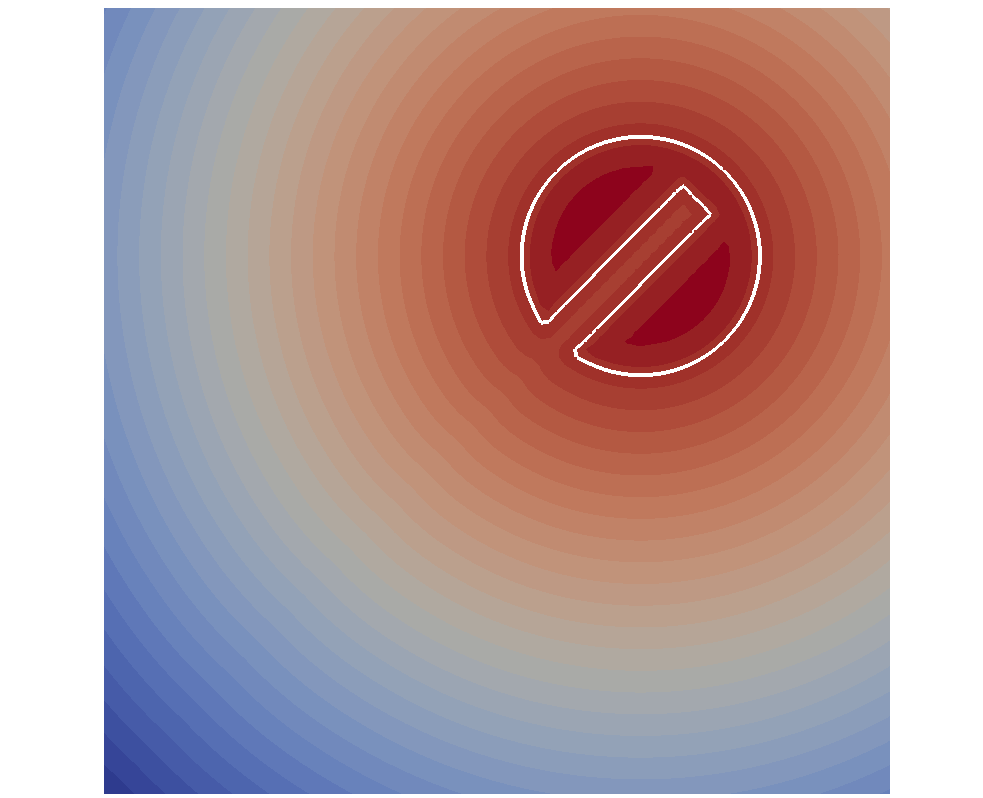} &
\includegraphics[width=0.3\textwidth]{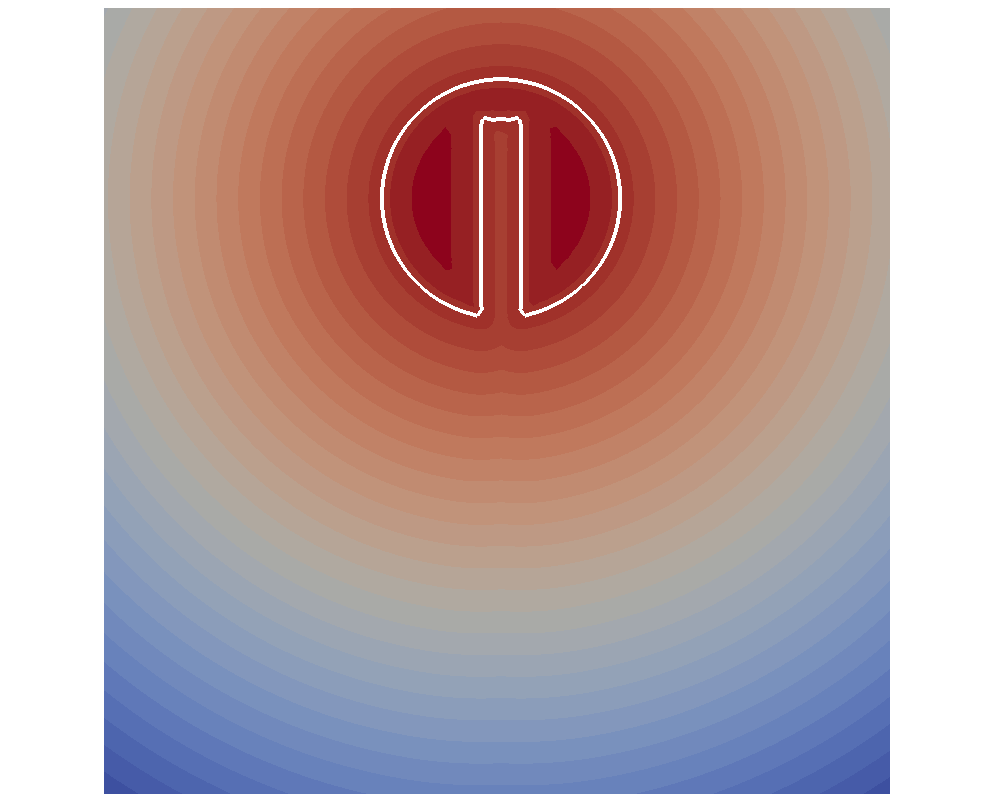}
\end{tabular}
\caption{\label{fig:zalesak}The signed-distance function and its zero contour for
Zalesak's disk problem at times $t=0$, $79$, $157$, $236$, $314$, $393$, $471$,
$550$, and $628$. From left to right, top to bottom. There is no significant
numerical diffusion: the corners do not get smeared out over time. Due to the
Lagrangian nature of the method, the results for $t=0$ and $t=628$ are even
identical.}
\end{figure}

\subsection{Comparison with Discontinuous Galerkin Methods}
Schroeder, Lehrenfeld, Linke, and Lube\autocite{schroeder2018} performed long-term
simulations with $t\in[0,26]$ of a two-dimensional flow on the domain $\Omega:=
(0,1)^2$ with periodic boundary conditions. The exact solution to their benchmark
problem reads:
\begin{equation}
\vv{u}_0(\vv{x}) =
\begin{pmatrix}\sin(2\pi x_1)\sin(2\pi x_2) \\ \cos(2\pi x_1)\cos(2\pi x_2)\end{pmatrix},
\quad
\vv{u}(\vv{x},t) =  e^{-8\pi^2\nu t}\vv{u}_0(\vv{x}),
\quad
\nu = 10^{-5}.
\end{equation}
This flow is dynamically unstable and small perturbations quickly lead to
chaotic motion. In numerical methods this will inevitably occur, the 
challenge is to minimise the rate at which the numerical solutions diverge.

In their paper they emphasise the importance of \emph{exactly enforcing} $\nabla
\cdot\vv{u}=0$ in numerical simulations. Methods that do not share this property,
e.\,g., finite element formulations based on the Taylor--Hood pair, lose 12
significant digits before reaching $t=2$. Schroeder, Lehrenfeld, Linke, and
Lube applied an exactly divergence free, eighth order, hybridised discontinuous
Galerkin formulation~(HDG8) to this problem. In their simulation the rate of
error increase was significantly smaller. Vortex methods also fulfil $\nabla
\cdot\vv{u}=0$ exactly, and we were kindly provided with the HDG8 simulation
results for a comparison.

The two-dimensional Navier--Stokes equations in their vorticity formulation
read:
\begin{equation}
\pdx{\omega}{t} + (\vv{u}\cdot\nabla)\omega = \nu\Delta\omega.
\end{equation}
Here, the velocity $\vv{u}=(u_1,u_2)^\top$ is the solution to the system:
\begin{equation}
\left\lbrace
\begin{aligned}
\partial_{x_1}u_1 + \partial_{x_2}u_2 &= 0, \\
\partial_{x_1}u_2 - \partial_{x_2}u_1 &= \omega.
\end{aligned}
\right.
\end{equation}
The solution to this system can be obtained from $\omega$ by first solving the
Poisson problem $-\Delta\psi = \omega$ for the stream function $\psi$ with
periodic boundary conditions and then setting $\vv{u}=(\partial_{x_2}\psi,
-\partial_{x_1}\psi)^\top$ as before. The vortex particle method discretises
this set of equations and proceeds in the following steps:
\begin{enumerate}
\item Let $t=0$ and initialise the particle field $\omega_h$ from $\omega_0 =
\mathrm{curl}\;\vv{u}_0 = \partial_{x_1}u_{0,2} - \partial_{x_2}u_{0,1}$ as
described in \cref{sec:particleapprox}. We place the particles at \emph{random
locations} inside the cells of the $h$-grid: the spectral accuracy of the
mid-point rule in this case would give a wrong picture of the method's
accuracy.
\item Repeat until $t=T$:
\begin{enumerate}
\item[2.1.] Compute $\omega_{\sigma,h}(\cdot,t) = A_h^{-1}\omega_h(t)$ in
$V_\sigma^n(\Omega)$, where the space $V_\sigma^n(\Omega)$ is  supplemented with
periodic boundary conditions.
\item[2.2.] Solve the Poisson problem $-\Delta\psi_{\sigma,h}(\cdot,t)=
\omega_{h,\sigma}(\cdot,t)$ in $V_\sigma^{n+2}(\Omega)$, again with periodic
boundary conditions. We use a standard Galerkin method for this.
\item[2.3.] Define $\vv{u}_{h,\sigma}(\cdot,t) := \bigl(\partial_{x_2}
\psi_{h,\sigma}(\cdot,t),-\partial_{x_1}\psi_{h,\sigma}(\cdot,t)\bigr)^\top$. 
\item[2.4.] Advance the following system of ODEs by one step $\Delta t$ in time
using, e.\,g., a Runge--Kutta method:
\begin{equation}
\left\lbrace
\begin{aligned}
\ddx{\vv{x}_i}{t}(t) &= \vv{u}_{h,\sigma}\bigl(\vv{x}_i(t),t\bigr), \\
\ddx{\omega_i}{t}(t) &= \nu\Delta\omega_{h,\sigma}\bigl(\vv{x}_i(t),t\bigr),
\end{aligned}
\right.\qquad i=1,\dotsc,N.
\end{equation}
\end{enumerate}
\end{enumerate}

Schroeder, Lehrenfeld, Linke, and Lube\autocite{schroeder2018} performed their
simulations on unstructured grids of sizes $\sigma\approx 0.25$ and $\sigma%
\approx 0.05$, together with a second order time discretisation and a fixed time-%
step of $\Delta t = 10^{-4}$. We perform our experiments with $\sigma\in\lbrace
\tfrac{1}{13},\tfrac{1}{30}\rbrace$, $h=\tfrac{\sigma}{2}$, i.\,e., $d=\tfrac{1}%
{2}$, and orders $n\in\lbrace4,6,8\rbrace$. The grid-sizes were chosen such that
for $\sigma=\tfrac{1}{13}$ the initial $L^2(\Omega)$-error in the velocity for
$n=6$ roughly equals that of the HDG8 computation on the coarse mesh. Similarly,
for $\sigma=\tfrac{1}{30}$ the initial error of the HDG8 scheme roughly equals
that of the eighth order vortex method. For the time discretisation we use
Verner's \enquote{most efficient} ninth order Runge--Kutta method~\autocite{%
verner2010} with a fixed time-step of $\Delta t = \tfrac{1}{32}$. This time-step
is more than 300 times larger than the one used for the HDG8 computations.

The results are depicted in \Cref{fig:schroeder-bench}. One clearly sees that the
vortex methods perform very much like the HDG8 schemes and that errors increase
at an equal rate. We conclude that vortex methods can compete with state-of-the-art
discontinuous Galerkin methods. At the same time, due to the high degree of
regularity of the ansatz spaces, very few degrees of freedom (DOF) are necessary.
In fact, in this particular case we have three degrees of freedom per particle
(two for $\vv{x}_i$ and one for $\omega_i$) and two per grid-node (one each for
$\psi_{h,\sigma}$ and $\omega_{h,\sigma}$); both numbers are independent of the
order $n$. For $\sigma= \tfrac{1}{30}$ one obtains $N_\text{DOF} = 2\times30^2+
3\times60^2 =12\,600$ compared to $23\,903$ non-eliminable DOFs for the HDG8 scheme.

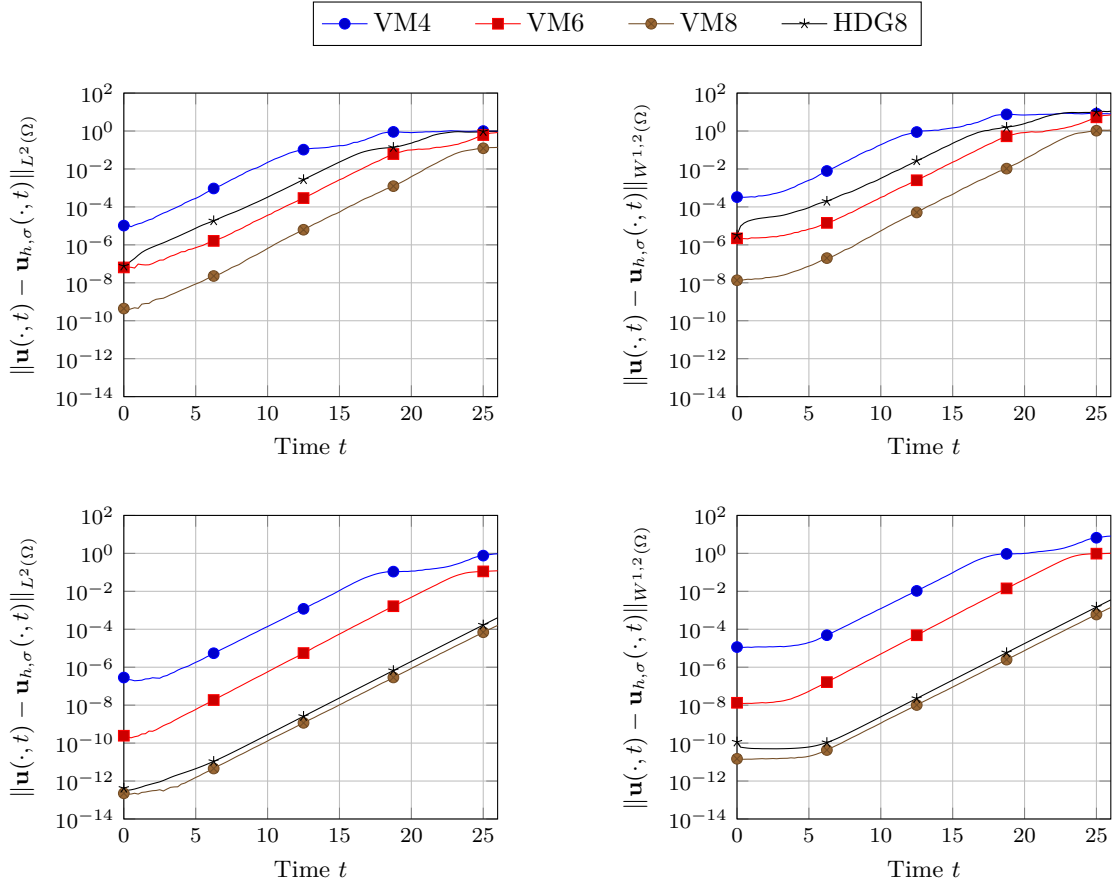
\begin{figure}
\centering
\begin{tikzpicture}
\begin{groupplot}
[
	small,
    group style={group size=2 by 2, horizontal sep=0.2\textwidth, vertical sep=0.1\textwidth},
	ymode = log,
    grid = major,
	xmin = {0},
	xmax = {26},
    ymin = {1e-14},
	ymax = {1e2},
    ylabel = {$\Vert\vv{u}(\cdot,t) - \vv{u}_{h,\sigma}(\cdot,t)\Vert_{L^2(\Omega)}$},
    xlabel = {Time $t$},
    ytick = {1e-14,1e-12,1e-10,1e-8,1e-6,1e-4,1e-2,1e0,1e2},
]
\nextgroupplot[mark repeat=25,mark phase=0,ylabel = {$\Vert\vv{u}(\cdot,t) - \vv{u}_{h,\sigma}(\cdot,t)\Vert_{L^2(\Omega)}$}]
\addplot table[x=time,y=L2Velocity]{vm4c-schroeder.dat};
\addplot table[x=time,y=L2Velocity]{vm6c-schroeder.dat};
\addplot table[x=time,y=L2Velocity]{vm8c-schroeder.dat};
\addplot table[x=time,y=L2Velocity]{hdg8c-schroeder.dat};
\nextgroupplot[mark repeat=25,mark phase=0,ylabel = {$\Vert\vv{u}(\cdot,t) - \vv{u}_{h,\sigma}(\cdot,t)\Vert_{W^{1,2}(\Omega)}$}]
\addplot table[x=time,y=H1Velocity]{vm4c-schroeder.dat};
\addplot table[x=time,y=H1Velocity]{vm6c-schroeder.dat};
\addplot table[x=time,y=H1Velocity]{vm8c-schroeder.dat};
\addplot table[x=time,y=H1Velocity]{hdg8c-schroeder.dat};
\nextgroupplot[mark repeat=25,mark phase=0,ylabel = {$\Vert\vv{u}(\cdot,t) - \vv{u}_{h,\sigma}(\cdot,t)\Vert_{L^2(\Omega)}$}]
\addplot table[x=time,y=L2Velocity]{vm4f-schroeder.dat};
\addplot table[x=time,y=L2Velocity]{vm6f-schroeder.dat};
\addplot table[x=time,y=L2Velocity]{vm8f-schroeder.dat};
\addplot table[x=time,y=L2Velocity]{hdg8f-schroeder.dat};
\nextgroupplot[mark repeat=25,mark phase=0,ylabel = {$\Vert\vv{u}(\cdot,t) - \vv{u}_{h,\sigma}(\cdot,t)\Vert_{W^{1,2}(\Omega)}$},
legend to name = SharedLegend, legend style = {legend columns=4,/tikz/every even column/.append style={column sep=0.5cm}}]
\addplot table[x=time,y=H1Velocity]{vm4f-schroeder.dat};
\addplot table[x=time,y=H1Velocity]{vm6f-schroeder.dat};
\addplot table[x=time,y=H1Velocity]{vm8f-schroeder.dat};
\addplot table[x=time,y=H1Velocity]{hdg8f-schroeder.dat};
\addlegendimage{red mark=*};
\addlegendentry{VM4};
\addlegendentry{VM6};
\addlegendentry{VM8};
\addlegendentry{HDG8};
\end{groupplot}
\path (group c1r1.north west) -- node[above=0.5cm]{\pgfplotslegendfromname{SharedLegend}} (group c2r1.north east);
\end{tikzpicture}
\caption{\label{fig:schroeder-bench}Evolution of the $L^2(\Omega)$- and $W^{1,2}
(\Omega)$-errors of the fourth, sixth, and eighth order vortex methods over time
in comparison to an eighth order discontinuous Galerkin scheme on coarse~(top)
and fine~(bottom) meshes.}
\end{figure}

\subsection{Application to a three-dimensional Flow}
As an example of a three-dimensional problem, we consider the Arnold--Beltrami--%
Childress flow~\cite[pp.~56ff]{majda2001} on the domain $\Omega=(0,2\pi)^3$ for $t\in[0,10]$:
\begin{equation}
\vv{u}_0(\vv{x}) =
\begin{pmatrix} \sin(x_3) + \cos(x_2) \\ \sin(x_1) + \cos(x_3) \\ \sin(x_2) + \cos(x_1)\end{pmatrix},
\quad
\vv{u}(\vv{x},t) =  e^{-\nu t}\vv{u}_0(\vv{x}),
\quad
\nu = 10^{-3}.
\end{equation}
This flow is one of the few known fully three-dimensional, analytic solutions to
the Navier--Stokes equations with periodic boundary conditions. Due to the larger
viscosity, this flow is only mildly unstable.

In three dimensional space, the vorticity formulation of the Navier--Stokes
equations reads:
\begin{equation}
\pdx{\gv{\omega}}{t} + (\vv{u}\cdot\nabla)\gv{\omega} =
\bigl(\boldnabla\vv{u}\bigr)\cdot\gv{\omega} + \nu\boldlaplace\gv{\omega}.
\end{equation}
Unlike in two dimensions, the vorticity $\gv{\omega}$ now also is a vector-%
valued quantity, and the equation is augmented with the so-called \emph{%
vortex stretching} term $\bigl(\boldnabla\vv{u}\bigr)\cdot\gv{\omega}$.
The velocity $\vv{u}$ can be obtained from the vortictiy $\gv{\omega}$ by
solving the system $\nabla\cdot\vv{u}=0$, $\nabla\times\vv{u}=\gv{\omega}$.
The vortex method discretises this set of equations analogously to the two-%
dimensional case:
\begin{enumerate}
\item Let $t=0$ and initialise the particle field $\gv{\omega}_h$ from
$\gv{\omega}_0 = \nabla\times\vv{u}_0$. We place the particles at \emph{random
locations} inside the cells of the $h$-grid.
\item Repeat until $t=T$:
\begin{enumerate}
\item[2.1.] Compute $\gv{\omega}_{\sigma,h}(\cdot,t) = A_h^{-1}\gv{\omega}_h(t)$
in $\left(V_\sigma^n(\Omega)\right)^3$, where the space $\left(V_\sigma^n(\Omega)
\right)^3$ is supplemented with periodic boundary conditions.
\item[2.2.] Solve the Poisson problem $-\boldlaplace\gv{\psi}_{\sigma,h}(\cdot,t)=
\gv{\omega}_{h,\sigma}(\cdot,t)$ in $\left(V_\sigma^{n+2}(\Omega)\right)^3$,
again with periodic boundary conditions. We use a standard Galerkin method for this.
\item[2.3.] Define $\vv{u}_{h,\sigma}(\cdot,t) := \nabla\times\gv{\psi}_{h,\sigma}(\cdot,t)$. 
\item[2.4.] Advance the following system of ODEs by one step $\Delta t$ in time
using, e.\,g., a Runge--Kutta method:
\begin{equation}
\left\lbrace
\begin{aligned}
\ddx{\vv{x}_i}{t}(t) &= \vv{u}_{h,\sigma}\bigl(\vv{x}_i(t),t\bigr), \\
\ddx{\gv{\omega}_i}{t}(t) &= 
\biggl[\bigl(\boldnabla\vv{u}_{h,\sigma}\bigr)\cdot
       \bigl(\nabla\times\vv{u}_{h,\sigma}\bigr)\biggr]
\bigl(\vv{x}_i(t),t\bigr) +
\nu\boldlaplace\gv{\omega}_{h,\sigma}\bigl(\vv{x}_i(t),t\bigr),
\end{aligned}
\right.\qquad i=1,\dotsc,N.
\end{equation}
\end{enumerate}
\end{enumerate}

We perform a convergence study using $\sigma\in[\tfrac{2\pi}{10},\tfrac{2\pi}{40}]$,
$h=\tfrac{\sigma}{2}$, i.\,e., $d=\tfrac{1}{2}$, $n\in\lbrace 4,6\rbrace$, and a
fixed time-step of $\Delta t = \tfrac{1}{25}$ using Verner's ninth order Runge--%
Kutta method.

For $\sigma=\frac{2\pi}{10}$ and $n=4$ a video of the evolving particle field at
25 steps-per-second was created.\footnote{\url{https://rwth-aachen.sciebo.de/s/5tueQcMJeqWjPut},
a temporary link for the preprint. Can be played using, e.\,g., the VLC Media
Player.} The reader is invited to take a look: while it is hard to measure the
beauty of a method or flow, one can clearly see that this flow is non-trivial and
that the particle method remains stable. This is also quantitatively confirmed in
\Cref{fig:abc-bench}, where the evolution of the $W^{1,2}(\Omega)$ velocity error
over time is shown: while for the coarse discretisations the error grows only
mildly over time, it stays essentially constant for the fine ones. In
\Cref{tab:abc-convergence} the errors for the various discretisations at final
time $T=10$ are shown. The results confirm that the methods are of order $n$.

\begin{figure}
\centering
\begin{tikzpicture}
\begin{groupplot}
[
    small,
    group style={group size=2 by 1, horizontal sep=0.2\textwidth},
    ymode = log,
    grid = major,
    xmin = {0},
    xmax = {10},
    ylabel = {$\Vert\vv{u}(\cdot,t) - \vv{u}_{h,\sigma}(\cdot,t)\Vert_{W^{1,2}(\Omega)}$},
    xlabel = {Time $t$},
]
\nextgroupplot[mark repeat = 25, mark phase = 0, ymin=1e-5, ymax=1e-1]
\addplot table[x=time,y=H1Velocity]{abcflow-vm4-10-20.dat};
\addplot table[x=time,y=H1Velocity]{abcflow-vm4-14-28.dat};
\addplot table[x=time,y=H1Velocity]{abcflow-vm4-20-40.dat};
\addplot table[x=time,y=H1Velocity]{abcflow-vm4-28-56.dat};
\addplot table[x=time,y=H1Velocity]{abcflow-vm4-40-80.dat};
\nextgroupplot[mark repeat = 25, mark phase = 0, ymin=1e-9, ymax=1e-3,
legend to name = SharedLegend, legend style = {legend columns=5,/tikz/every even column/.append style={column sep=0.5cm}}]
\addplot table[x=time,y=H1Velocity]{abcflow-vm6-10-20.dat};
\addplot table[x=time,y=H1Velocity]{abcflow-vm6-14-28.dat};
\addplot table[x=time,y=H1Velocity]{abcflow-vm6-20-40.dat};
\addplot table[x=time,y=H1Velocity]{abcflow-vm6-28-56.dat};
\addplot table[x=time,y=H1Velocity]{abcflow-vm6-40-80.dat};
\addlegendimage{red mark=*};
\addlegendentry{$\sigma=\frac{2\pi}{10}$};
\addlegendentry{$\sigma=\frac{2\pi}{14}$};
\addlegendentry{$\sigma=\frac{2\pi}{20}$};
\addlegendentry{$\sigma=\frac{2\pi}{28}$};
\addlegendentry{$\sigma=\frac{2\pi}{40}$};
\end{groupplot}
\path (group c1r1.north west) -- node[above=0.5cm]{\pgfplotslegendfromname{SharedLegend}} (group c2r1.north east);
\end{tikzpicture}
\caption{\label{fig:abc-bench}The evolution of the $W^{1,2}(\Omega)$ velocity
error for the ABC flow problem. Note the different scales for order $n=4$ on the
left and order $n=6$ on the right. The error increases only mildly over time for
the coarse discretisations and stays almost unchanged for the fine ones.}
\end{figure}
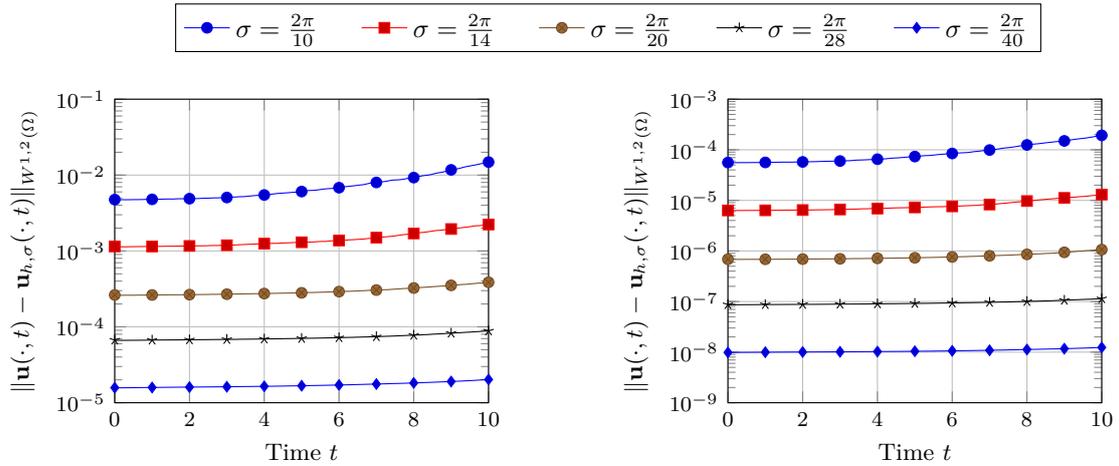

\begin{table}
\centering
\renewcommand{\arraystretch}{1.1}
\begin{tabular*}{0.717\textwidth}{lcccc}
\toprule
$n = 4.$ &
$\hspace{0.6cm}\Vert\vv{u}-\vv{u}_{h,\sigma}\Vert_{L^2(\Omega)}$      &  EOC &
$\hspace{0.8cm}\Vert\vv{u}-\vv{u}_{h,\sigma}\Vert_{W^{1,2}(\Omega)}$  &  EOC \\\midrule
$\sigma=\tfrac{2\pi}{10}$ & \sci{3.60}{-3} & ---  & \sci{1.48}{-2} &  ---  \\
$\sigma=\tfrac{2\pi}{14}$ & \sci{4.12}{-4} & 6.44 & \sci{2.23}{-3} & 5.62  \\
$\sigma=\tfrac{2\pi}{20}$ & \sci{5.17}{-5} & 5.82 & \sci{3.86}{-4} & 4.92  \\
$\sigma=\tfrac{2\pi}{28}$ & \sci{1.01}{-5} & 4.85 & \sci{8.84}{-5} & 4.38  \\
$\sigma=\tfrac{2\pi}{40}$ & \sci{2.45}{-6} & 3.97 & \sci{2.02}{-5} & 4.14  \\\bottomrule
\end{tabular*}

\vspace{1cm}

\begin{tabular*}{0.717\textwidth}{lcccc}
\toprule
$n = 6.$ &
$\hspace{0.6cm}\Vert\vv{u}-\vv{u}_{h,\sigma}\Vert_{L^2(\Omega)}$      &  EOC &
$\hspace{0.8cm}\Vert\vv{u}-\vv{u}_{h,\sigma}\Vert_{W^{1,2}(\Omega)}$  &  EOC \\\midrule
$\sigma=\tfrac{2\pi}{10}$ & \sci{4.63}{-5} & ---  & \sci{1.93}{-4} &  ---  \\
$\sigma=\tfrac{2\pi}{14}$ & \sci{2.44}{-6} & 8.75 & \sci{1.30}{-5} &  8.02 \\
$\sigma=\tfrac{2\pi}{20}$ & \sci{1.36}{-7} & 8.10 & \sci{1.06}{-6} &  7.03 \\
$\sigma=\tfrac{2\pi}{28}$ & \sci{1.12}{-8} & 7.42 & \sci{1.14}{-7} &  6.63 \\
$\sigma=\tfrac{2\pi}{40}$ & \sci{1.30}{-9} & 6.04 & \sci{1.24}{-8} &  6.22 \\\bottomrule
\end{tabular*}
\caption{\label{tab:abc-convergence}The $L^2(\Omega)$ and $W^{1,2}(\Omega)$
velocity errors at time $T=10$ for the ABC flow problem at different
discretisation sizes $\sigma$ and orders $n=4$ (top) and $n=6$ (bottom).
The empirical orders of convergence (EOC) approach the theoretical ones.}
\end{table}

\section{Outlook}\label{sec:outlook}
The regularisation scheme considered in this article uses a uniform, non-adaptive
Cartesian grid. As shown in theory and practice, this scheme is asymptotically
optimal for convection dominated flows if both the initial data and the velocity
field are sufficiently smooth. In these cases vortex methods in particular can
compete with discontinuous Galerkin methods.

Many flows of practical interest, however, feature steep gradients, leading to
similarly steep gradients in the solution. This is especially true for turbulent
flows. If applied to such flows, the uniform regularisation scheme presented in
this work requires very small choices of $d$ for $A_h^{-1}$ to remain well-%
conditioned, thereby reducing its efficiency. 

The particles naturally adapt to such flow fields. In fact, particles cluster
where steep gradients occur, while the particle field \enquote{thins out} in the
areas where gradients get flat. To see this, let us reconsider the analytic
solution of the linear advection equation: $u(\vv{x},t)=
u_0\bigl(\traj{t}{0}(\vv{x})\bigr)$. A simple application of the chain rule yields:
\begin{equation}
\nabla u(\vv{x},t) = \nabla u_0\bigl(\traj{t}{0}(\vv{x})\bigr)\,\cdot\,\boldnabla\traj{t}{0}(\vv{x}).
\end{equation}
Therefore, steep gradients that were not already present in $u_0$ can only arise
if $\boldnabla\traj{t}{0}$ is \enquote{large}. Let $\vv{x}_i$ and $\vv{x}_j$
denote two particles that are close to one another at time $t$ and let $\vv{z}:=
(\vv{x}_j(t) - \vv{x}_i(t))/|\vv{x}_j(t) - \vv{x}_i(t)|$. We then have
approximately:
\begin{equation}
\pdx{\traj{t}{0}}{\vv{z}}\bigl(\vv{x}_i(t)\bigr) \approx
\frac{\traj{t}{0}\bigl(\vv{x}_j(t)\bigr)-\traj{t}{0}\bigl(\vv{x}_i(t)\bigr)}{|\vv{x}_j(t)-\vv{x}_i(t)|} =
\frac{\vv{x}_j(0)-\vv{x}_i(0)}{|\vv{x}_j(t)-\vv{x}_i(t)|}.
\end{equation}
Thus, derivatives of $\traj{t}{0}$ get large when particles are close together
that previously were not. Conversely, the derivatives are small if particles
move away from one another. It therefore would make sense to also adapt the
Ansatz spaces for the regularisation scheme accordingly: the resolution should
be coarse where there are few particles and fine where there are many. This
would also ensure that the operator $A_h^{-1}$ corresponding to these spaces
remains well-defined. In the context of splines it would be interesting to
develop methods based on wavelets to achieve this adaption of the Ansatz spaces.

One can also assign new quadrature weights to a given particle field. For this
one subdivides the domain $\Omega$ into new cells $Q_i$, such that each contains
exactly one particle. Afterwards, each particle is assigned the weight $w_i=
\meas_\spdim(Q_i)$. We believe that this also makes our method interesting for
scattered data approximation.

Another topic that was not covered in detail here is time discretisation. In
this work we simply used standard Runge--Kutta methods. Given the apparent
importance of exactly enforcing $\nabla\cdot\vv{u}=0$ when solving the
incompressible Navier--Stokes equations, it would make sense to use
volume preserving schemes for solving the ODEs $\dot{\vv{x}}_i(t)=
\vv{u}\bigl(\vv{x}_i(t),t\bigr)$, $i=1,\dotsc,N$. Maybe this would even further
improve the long-term accuracy of particle approximations when applied to such
problems.

\appendix

\section*{Appendix}

\subsection*{Modification of the original Proof of $L^\infty(\Omega)$-Stability\autocite{douglas1974,crouzeix1987}}

To establish the boundedness of $A_h^{-1}$ as an operator $V_\sigma^{-n,p}(\Omega)
\to V_\sigma^{n,p}(\Omega)$, $p\neq 2$, it suffices to consider functionals of the
form $\idx{\Omega}{}{fv_\sigma}{x}$, $f\in L^p(\Omega)$. We thus let $f\in L^p(\Omega)$,
fix a cell $Q_j^\sigma\in \Omega$ and define $f_j:=f$ on $Q_j^\sigma$ and $f_j:=0$
on $\Omega\setminus Q_j^\sigma$. We will establish that $f_{j,\sigma}:=A_h^{-1}f_j$
decays at an exponential rate away from $Q_j^\sigma$.

We have for all $v_\sigma\in V_\sigma^n(\Omega)$ that vanish on $Q_j^\sigma$:
$\langle A_hf_{j,\sigma}, v_\sigma\rangle = \idx{Q_j^\sigma}{}{fv_\sigma}{x}=0$.
We now construct such a function $v_\sigma$. Let us define the neighbourhoods of
$Q_j^\sigma$ as $D_{j,0}:=\emptyset$, $D_{j,1}:=Q_j^\sigma$, and for all other
$k\in\mathbb{N}$:
\begin{equation}
D_{j,k}:=\bigcup\bigl\lbrace Q_l^\sigma\in\Omega:\ 
|\text{centre}(Q_l^\sigma)-\text{centre}(Q_j^\sigma)|_\infty < k\sigma
\bigr\rbrace,
\end{equation}
where $|\cdot|_\infty$ denotes the Manhattan distance norm on $\wholespace$. We
define $v_\sigma$ as follows: for $k\geq n$ we let $v_\sigma = f_{j,\sigma}$ on
$\Omega\setminus D_{j,k}$ and set the remaining B-spline coefficients to zero.
It follows that $v_\sigma = 0$ on $D_{j,k-(n-1)}$. Thus, denoting the quadrature
error $e:=\sum_{i=1}^{N}w_if_{j,\sigma}(\vv{x}_i)v_\sigma(\vv{x}_i)-\idx{\Omega}%
{}{f_{j,\sigma}v_\sigma}{\vv{x}}$, we have:
\begin{multline}
0 = \sum_{i=1}^{N}w_if_{j,\sigma}(\vv{x}_i)v_\sigma(\vv{x}_i) =
\idx{\Omega}{}{f_{j,\sigma}v_\sigma}{\vv{x}} + 
\left(\sum_{i=1}^{N}w_if_{j,\sigma}(\vv{x}_i)v_\sigma(\vv{x}_i)-\idx{\Omega}%
{}{f_{j,\sigma}v_\sigma}{\vv{x}}\right)\\
\iff 
\idx{\Omega\setminus D_{j,k}}{}{f_{j,\sigma}v_\sigma}{\vv{x}} + e
= \idx{D_{j,k}\setminus D_{j,k-(n-1)}}{}{f_{j,\sigma}v_\sigma}{\vv{x}} \\
\iff
\idx{\Omega\setminus D_{j,k}}{}{f_{j,\sigma}^2}{\vv{x}} + e
= \idx{D_{j,k}\setminus D_{j,k-(n-1)}}{}{f_{j,\sigma}v_\sigma}{\vv{x}}.
\end{multline}
Due to the stability of the B-spline basis, we have $\Vert v_\sigma
\Vert_{L^2(D_{j,k}\setminus D_{j,k-(n-1)})}\leq C_1\Vert f_{j,\sigma}
\Vert_{L^2(D_{j,k}\setminus D_{j,k-(n-1)})}$. The right hand side
of the last equation can thus be bounded from above by 
$C_1\Vert f_{j,\sigma}\Vert_{L^2(D_{j,k}\setminus D_{j,k-(n-1)})}^2$.
On the left we insert the quadrature error bound~\eqref{eqn:quaderror}
with error constant $C_2$ and obtain:
\begin{multline}\label{eqn:omfg}
\Vert f_{j,\sigma}\Vert_{L^2(\Omega\setminus D_{j,k})}^2 -
C_2h|f_{j,\sigma}v_\sigma|_{W^{1,1}(\Omega)}
\leq C_1 \Vert f_{j,\sigma}\Vert_{L^2(D_{j,k}\setminus D_{j,k-(n-1)})}^2 \\
\iff
\Vert f_{j,\sigma}\Vert_{L^2(\Omega\setminus D_{j,k})}^2 -
C_2h|f_{j,\sigma}^2|_{W^{1,1}(\Omega\setminus D_{j,k})}
\leq \\ 
C_1\Vert f_{j,\sigma}\Vert_{L^2(D_{j,k}\setminus D_{j,k-(n-1)})}^2 +
C_2h|f_{j,\sigma}v_\sigma|_{W^{1,1}(D_{j,k}\setminus D_{j,k-(n-1)})}.
\end{multline}
Now, using an inverse inequality, with constant $C_3$:
\begin{equation}
C_2h|f_{j,\sigma}^2|_{W^{1,1}(\Omega\setminus D_{j,k})}
\leq C_2C_3\frac{h}{\sigma}\Vert f_{j,\sigma}^2\Vert_{L^1(\Omega\setminus D_{j,k})}
= dC_2C_3\Vert f_{j,\sigma}\Vert^2_{L^2(\Omega\setminus D_{j,k})}.
\end{equation}
The left side of the last inequality in~\eqref{eqn:omfg} can thus be bounded
from below by $C_4\Vert f_{\sigma,j}\Vert_{L^2(\Omega\setminus D_{j,k})}^2$,
with $C_4=(1-dC_2C_3)$. For $d$ small enough we have $C_4>0$. Similarly, the
right side can be bounded by $C_5\Vert f_{j,\sigma}\Vert_{L^2(D_{j,k}\setminus D_{j,{k-(n-1)}})}$,
with error constant $C_5=C_1(1+dC_2C_3)$. Thus, with $C_6=\tfrac{C_5}{C_4}$:
\begin{equation}
\Vert f_{j,\sigma}\Vert_{L^2(\Omega\setminus D_{j,k})}^2 \leq C_6
\Vert f_{j,\sigma}\Vert_{L^2(D_{j,k}\setminus D_{j,k-(n-1)})}^2.
\end{equation}
But we have $D_{j,k}\setminus D_{j,k-(n-1)} = (\Omega\setminus D_{j,k-(n-1)}) \setminus
(\Omega\setminus D_{j,k})$, and thus we obtain:
\begin{equation}
\Vert f_{j,\sigma}\Vert_{L^2(\Omega\setminus D_{j,k})}^2 \leq 
\underbrace{\frac{C_6}{1+C_6}}_{=:C_7}
\Vert f_{j,\sigma}\Vert_{L^2(\Omega\setminus D_{j,k-(n-1)})}^2,
\end{equation}
where we obviously have $0<C_7<1$. For large values of $k$ this argument can now
now be repeated on the right hand side, and we obtain:
\begin{equation}
\Vert f_{j,\sigma}\Vert_{L^2(\Omega\setminus D_{j,k})}^2 \leq 
C_7^{\lfloor\frac{k}{n}\rfloor}\Vert f_{j,\sigma}\Vert_{L^2(\Omega)}^2.
\end{equation}
This is the desired exponential decay. From here the proof is identical to the
original ones.\autocite{douglas1974,crouzeix1987}

\subsection*{Source Code for the Initial Data of Zalesak's Disk}
\begin{lstlisting}
//
// Signed distance function for Zalesak's disk.
//
real initial_data( point x )
{
    using std::min;

    constexpr point A { 0, 0, 0 };
    constexpr point B { 0, 0.25, 0 };
    constexpr point C {-0.025, 0.35, 0 };
    constexpr point D { 0.025, 0.35, 0 };
    constexpr real  R  { 0.15 };
    const     real phi { std::asin(0.025/0.15) };
    constexpr point E {-0.025, 0.25 - R*std::cos(phi), 0 }; 
    constexpr point F { 0.025, 0.25 - R*std::cos(phi), 0 }; 

    auto is_in_cone = [phi]( point x ) noexcept -> bool
    {
        return std::acos(scal_prod(x-B,A-B)/( (x-B).r() * (A-B).r() )) < phi;
    };

    auto is_in_box = [C,D]( point x ) noexcept -> bool
    {
        return (x.x > C.x) && (x.x < D.x) && (x.y < D.y);
    };

    if ( length(x-B)<R )
    {
        if ( is_in_box(x) )
        {
            if ( x.y < F.y )
            {
                return -min(length(x-E),length(x-F));
            }
            else
            {
                return -min( min(x.x-E.x,F.x-x.x), C.y-x.y );
            }
        }
        else
        {
            if ( x.y > D.y )
            {
                real dist = min( min(length(x-C),length(x-D)), R-length(x-B) );
                if ( C.x < x.x && x.x < D.x ) return min(dist,x.y-C.y);
                else return dist;
            }
            else
            {
                if ( x.x < B.x )
                {
                    return min( C.x - x.x, R - length(x-B) );
                }
                else
                {
                    return min( x.x - D.x, R - length(x-B) );
                }
            }
        }
    }
    else
    {
        if ( is_in_cone(x) )
        {
            return -std::min(length(x-E),length(x-F));
        }
        else
        {
            return -((x-B).r() - R);
        }
    }
}
\end{lstlisting}

\printbibliography
\end{document}